\documentclass[review,onefignum,onetabnum]{siamart190516}
\usepackage{latexsym,amsmath,amssymb}
\usepackage{color}\usepackage{mathrsfs}

\usepackage{lipsum}
\usepackage{amsfonts}
\usepackage{graphicx}
\usepackage{epstopdf}
\usepackage{algorithmic}
\ifpdf
  \DeclareGraphicsExtensions{.eps,.pdf,.png,.jpg}
\else
  \DeclareGraphicsExtensions{.eps}
\fi

\newsiamremark{remark}{Remark}
\newsiamremark{hypothesis}{Hypothesis}
\crefname{hypothesis}{Hypothesis}{Hypotheses}
\newsiamthm{claim}{Claim}

\headers{On the exponential  stability of a stratified flow to the 2D IDEAL MHD equations with  damping}{Yi Du, Wang Yang, and Yi Zhou}

\title{On the exponential  stability of a stratified flow to the 2D IDEAL MHD equations with  damping\thanks{Submitted to the editors DATE.}}

\author{Yi Du\thanks{Department of Mathematical Sciences, Jinan University, Guangzhou 510632, China. 
  (\email{duyidy@gmail.com}, \email{duyidy@jnu.edu.cn} ).}
\and Wang Yang\thanks{Department of Mathematical Sciences, Jinan University, Guangzhou 510632, China. 
  (\email{yang@stu2017.jnu.edu.cn}).}
\and Yi Zhou\footnotemark[3]\thanks{School of Mathematical Sciences, Fudan University, Shanghai 200433, China. Department of Mathematical Sciences, Jinan University, Guangzhou 510632, China. 
  (\email{yizhou@fudan.edu.cn} )}
}

\usepackage{amsopn}

\ifpdf
\hypersetup{
  pdftitle={ On the exponential  stability of a stratified flow to the 2D IDEAL MHD equations with  damping},
  pdfauthor={Yi Du, Wang Yang, and Yi Zhou}
}
\fi

\begin{document}

\maketitle

\begin{abstract}
  We study the stability of a type of stratified flows of the two dimensional inviscid incompressible MHD equations with velocity damping. The exponential stability for the perturbation near certain stratified flow is investigated in  a strip-type area $\mathbb{R}\times[0,1]$. Although  the magnetic filed potential is governed by a transport equation, by using the algebraic structure of the incompressible condition,  it turns out that the linearized MHD equations around the given stratified flow retain a non-local damping mechanism. After carefully analyzing the non-linear structure and introducing some suitable weighted energy norms, we get the exponential stability by combining the exponential decay in time in the lower order energy with that in the  high order energy.
\end{abstract}



\section{Introduction}
During the past decades, incompressible fluid equations have received much attention, because they are with mathematical challenge, and present many interesting phenomena, particularly the phenomena of the stability or instability related to the  longtime
behavior of solutions \cite{drazin}.
In recent years, researchers have discovered numerous new interesting phenomena, such as the stability  for the  planar Couette-flow
 in 2D Euler equations \cite{Masmoudi2013-2} and for  the incompressible
Navier-Stokes equations ($n=2,3$) with high Reynold's number \cite{Masmoudi2013-1}. To the
half space case in $\mathbb{R}^2$,  with  the assumption that the initial perturbation is small and periodic,  the system is  stable
even if the basic flow is large \cite{kozono,Romanov}. The stability for the compressible case with small Couette-type flow  is discussed in \cite{Kagei}. Moreover, as to different Mach numbers,  the stability  or instability for a small steady Poiseuille-type flow in a layer area in $\mathbb{R}^2$ have been discussed in \cite{Kagei1} and  \cite{Kagei2} respectively.
For more relevant results on this topic, readers can  see  \cite{giga-liu,jiang fei1,jiang fei3,guo2,guo3,guo4} and the references therein.

Magneto-Hydrodynamics (MHD) equations describe the motion of an electrically conducting fluid in the presence of the magnetic field, underlying many physical phenomena  such as the geomagnetic dynamo in geophysics and solar winds and
solar flares in astrophysics (\cite{davidson}). Mathematically, the  MHD equations  are extremely difficult to analyze due to the analogous  nonlinear structure and the strong nonlinear coupling with the incompressible Navier-Stokes equations.

In this paper, we shall study the exponential stability for a type of stratified  flow to the following 2D  MHD equations with damping,
 namely,
\begin{equation}\label{1.1}
\begin{cases}
\partial_t U +(U\cdot \nabla)U+\kappa U+\nabla P=(B\cdot\nabla)B,\quad (x,y,t)\in \Omega \times \mathbb{R}^+,\\
 \partial_t B +(U\cdot \nabla)B=(B\cdot\nabla)U,\\
 \nabla\cdot U=\nabla \cdot B=0,
\end{cases}
\end{equation}
where $U=(U_1,U_2)^T$, $B=(B_1,B_2)^T$ and the scalar  $P$ are the velocity,  magnetic field and the pressure of the fluid respectively. $\kappa$ is  a positive  constant. Here and hereafter, we denote $\nabla=(\partial_x,\partial_y)^T$
and the  layer area
\begin{equation}\label{1.2}
\Omega=\mathbb{R}\times [0,1].
\end{equation}

Noting the condition $\nabla\cdot B=0$ and in 2D, we take
\begin{equation}\label{1.3}
B=\nabla^\bot\Phi\triangleq(\partial_y\Phi,-\partial_x\Phi),
\end{equation}
with $\Phi$ as a scalar function. The system \eqref{1.1} is equivalent to
\begin{equation}\label{1.4}
\begin{cases}
\partial_t \Phi +(U\cdot \nabla)\Phi=0,\\
\partial_t U +(U\cdot \nabla)U+\kappa U+\nabla P=\nabla\cdot (\nabla^\bot \Phi\otimes \nabla^\bot \Phi),\quad (x,t)\in \Omega\times \mathbb{R}^+,\\
 \nabla\cdot U=0.
\end{cases}
\end{equation}

For completeness, we recall some related efforts  on the global regularity problem for the 2D MHD system first. The dissipative and resistive MHD system has been well studied \cite{Temam}.
The dissipative and non-resistive MHD system  has been first  studied by  Lin, Xu and Zhang \cite{zhang-lin-xu},
then  the proof of \cite{zhang-lin-xu} has been significantly simplified by Zhang \cite{zhang ting2}  later. The inviscid and resistive case  or mixed partial dissipation and partial magnetic diffusion 2D MHD system can be found in the series work of Cao and Wu, such as \cite{cao-wu1} and the references therein.
Considering the ideal MHD system (i.e. inviscid and non-resistive), it will bring extra difficulty to bound the vortex stretching type terms or the nonlinear coupled terms  even in 2D. Using  the Elasser's variable $Z^\pm= U\pm B$,
Bardos, Sulem and Sulem \cite{bardos} prove the global existence of the classical solution when the initial data is close
to the non-zero constant equilibrium state $(\bar{B}, 0)$. Interested readers may see more results in \cite{zhouyi1, wu-wu-xu,hu-lin,xiang-wu-zhang,zhang ting1,Yamazaki1,zhou-fan,zhouyi2} et. al.

 In this paper, we consider exponential stability for the  MHD system \eqref{1.4} with the  stratified flow of the form $\bar{U}(t,x,y)=(\Psi(t, y),0)$ and the magnetic potential $\bar{\bar{\Phi}}(t,x,y)=\bar{\Phi}(t,y)+C_0x$. Here $(x,y)\in\Omega$ and $C_0$ is a non-zero constant. It turns out that  functions $(\Psi(t,y),\bar{\Phi}(t,y))$  solve a 1D wave equation on the bounded interval $y\in[0,1]$, which is similar  to the case of visco-elasticity system given by Endo, Giga, Gotz and Liu \cite{giga-liu}.
We prove that if initially the  equilibrium   stratified  flow is sufficiently smooth, then the pair of solutions $(\bar{U},\bar{\bar{\Phi}})$ is exponential  stable as time tends to infinity.

 Denote
\begin{equation}\label{1.5}
u=U-\bar{U}, \quad \rho=\Phi-\bar{\bar{\Phi}},
\end{equation}
then the perturbed system \eqref{1.4} reads as
 \begin{equation}\label{1.6}
\partial_{t}\rho+u_{2}\partial_{y}\Bar{\Phi}+C_0u_{1}+(u\cdot\nabla)\rho+\Psi\partial_{x}\rho=0,
\end{equation}
and
\begin{multline}\label{1.7}
\partial_{t}u+(u\cdot\nabla)u+(u_{2}\partial_{y}\Psi,0)^T+\kappa u+\Psi\partial_{x}u+\nabla P\\
=(\nabla^{\bot}\rho\cdot\nabla)\nabla^{\bot}\rho-(\partial_{x}\rho\partial_{y}^{2}\Bar{\Phi},0)^T
+(\partial_{y}\Bar{\Phi}\partial_{xy}\rho,-\partial_{y}\Bar{\Phi}\partial_{x}^{2}\rho)^T-C_0(\partial_{y}^{2}\rho,-\partial_{xy}\rho)^T.
\end{multline}
Here and hereafter, $A^T$ means the transpose of the vector $A$.

We take the rigid boundary conditions for system \eqref{1.6}-\eqref{1.7} as
\begin{equation}\label{1.8}
\begin{cases}
\lim\limits_{|x|\rightarrow +\infty } u(x,y, t)=0, \quad u_2(x,y,t)|_{y=0,1}=0,\\
(\partial_yu_1-\partial_x u_2)(x,y,t)|_{y=0,1}=0, \quad \partial_y \Psi(t,y)|_{y=0,1}=0.
\end{cases}
\end{equation}

Now,  our main theorem states as follows:
\begin{theorem}
Let $\epsilon_0>0$ be a small constant and $k\geq 7$. For the system \eqref{1.5}-\eqref{1.8},   taking the initial equilibrium $(\Psi,\partial_t\Psi, \bar{\Phi})|_{t=0}= (\Psi_0(y), \Psi_1(y),\bar{\Phi}_0(y))\in H^{k+3}$, and the initial  perturbation as $u_0=(u_{10}, u_{20})$ and $\rho_0$,   if the following conditions  hold:
\begin{equation}\label{1.9}
\begin{cases}
\partial_y^{2n-1}\rho_0(x,y)|_{y=0,1}=\partial_y^{2n-1}\Psi_0(x,y)|_{y=0,1}=\partial_y^{2n-1}\bar{\Phi}_0(x,y)|_{y=0,1}=0,\\
\int_0^1 u_{10}(x,y ) dy=\int_0^1 \rho_0(x,y ) dy=0,
\end{cases}
\end{equation}
and
\begin{equation}\label{1.10}
\|u_0\|_{H^k(\Omega)}+\|\rho_0\|_{H^{k+1}(\Omega)} \lesssim \epsilon_0,
\end{equation}
then the equilibrium flow $(\bar{U},\bar{\bar{\Phi}})$ of the system \eqref{1.6}-\eqref{1.8} is  "exponentially stable"
\footnote{
Considering the abstract differential system $\dot{X}=g(X,t)$, there is a norm $\|\cdot\|_{X}$ and the solution $X(t)$ is obtained. Meanwhile, if there exists a positive constant  $a$,  satisfying $\|X(t)\|_{X}\lesssim \|X(t_0)\|_{X}e^{-a(t-t_0)}$ for $t\geq t_0$, we call that the solution $X(t)$ is exponentially stable in the sense of $\|\cdot\|_{X}$. }
 in the space of $H^k(\Omega)$.
\end{theorem}

 One main difficulty in studying the system \eqref{1.6}-\eqref{1.8} is in the variable $\rho$, which is governed by a transport equation.  It is a great challenging problem to establish the decay estimate to get our main result.  Besides, due to the coupling terms there is a loss of derivative.
 Recently, in \cite{cardoba2}, authors study the stability property for the  2D incompressible  Boussinesq equations with damping, where they introduce an interesting weighted estimate to handle the loss of derivative's problem.
We can borrow the ideas from \cite{cardoba2} to study our case.

Let us point out some technical difficulties in our proof. First, we consider a strongly nonlinear coupled system, which is much more complicated  than the linear coupled system \cite{cardoba2}. Besides,  since we work in a sobolev space frame with high order derivatives, it needs a careful analysis of boundary  conditions (see Section 2 for details).
To get decay estimates for the perturbation of magnetic potential $\rho$, it needs a suitable background magnetic field, and this thought is inspired by the heuristic work of \cite{bardos}. Actually, the linearized system turns out to be  a 1D wave equation with damping on the bounded interval $[0,1]$, then we can get the exponential decay when we eliminate the zero mode of $u_1$ and $\rho$. Here we do not need the smallness for the  equilibrium flow.  More precisely, due to the damped wave equations  on a 1D bounded interval, the solution  possesses an exponential decay. With this decay rate, we can get our desired results without the smallness assumption of the background.

The strategy to prove our results is in the following two folds. First,  the basic observation is the stratified  flow actually plays a role of stability. Based on this fact, we exploit the structure of the system, and get  good enough decay rates for both the perturbation and the basic flow.   Second, we use the weighted energy estimates to overcome the loss of derivatives.  In our proof, we use some ideas from \cite{giga-liu,wu-wu-xu,cardoba2}.

This paper is organized as follows: in Section 2, we shall present some preliminaries and give the estimates for the basic flow. Moreover, the boundary  conditions and our function space will also be presented.
In Section 3, the energy estimates will be given. In Section 4, we shall give the decay estimates, and then the main theorem will be proved in  Section 5.
Throughout the paper, we sometimes use the notation $A\lesssim B$ as an equivalent notation to $A \leq CB$ with an uniform constant $C$.

\section{Preliminaries}
We begin with a  well-known Sobolev inequality, see e.g. \cite{adams}.
\begin{lemma}
Let $\Omega_0\subset \mathbb{R}^n$ and $s\in \mathbb{N}$,  for $(f,g)\in H^s(\Omega_0)$  the following estimate holds:
\begin{equation}\label{2.1}
\|fg\|_{H^s}\lesssim \|f\|_{H^s}\|g\|_{L^\infty}+\|g\|_{H^s}\|f\|_{L^\infty}.
\end{equation}
\end{lemma}

\subsection{Estimates for the  basic flow}
Since we take the basic flows as $\bar{U}=(\Psi(t,y),0) $ and  $\bar{\bar{\Phi}}(x,y,t)=\bar{\Phi} (t,y)+C_0x$, then we have
\begin{equation}\label{2.2}
\begin{cases}
\partial_t\bar{U}+(\bar{U}\cdot\nabla)\bar{U}+\kappa\bar{U}+\nabla{\bar{P}}=\nabla\cdot{(\nabla^{\bot}\bar{\bar{\Phi}}\otimes\nabla^{\bot}\bar{\bar{\Phi}})},\\
\partial_t\bar{\bar{\Phi}}+(\bar{U}\cdot\nabla)\bar{\bar{\Phi}}=0.
\end{cases}
\end{equation}
Moreover, the system \eqref{2.2} can be simplified  as
\begin{equation}\label{2.3}
\begin{cases}
\partial_t\Psi+\kappa\Psi+\partial_{x}\bar{P}=-C_{0}\partial_{y}^{2}\bar{\Phi},\quad y\in[0,1],\\
\partial_{y}\bar{P}=0,\\
\partial_t\bar{\Phi}+C_{0}\Psi=0.
\end{cases}
\end{equation}

From $\eqref{2.3}_1$ and $\eqref{2.3}_2$ , we conclude that $\partial_{x}\bar{P}$ only depends on the variable $t$. We write $\partial_{x}\bar{P}=h(t)$ as in \cite{giga-liu} for a well picked function $h(t)$, which  does not play essential role in our proof. For simplicity, we take $\partial_{x}\bar{P}=h(t)=0$. And then from $\eqref{2.3}$, we have
\begin{equation}\label{2.4}
\begin{cases}
\partial_{t}^{2}\bar{\Phi}+\kappa\partial_{t}\bar{\Phi}-C_{0}^2\partial_{y}^{2}\bar{\Phi}=0,\\
\Psi=-\frac{1}{C_{0}}\partial_{t}\bar{\Phi}.
\end{cases}
\end{equation}

By using  the boundary condition \eqref{1.8},
we have $\partial_y\Psi(t,y)|_{y=0,1}=\partial_{ty}\Psi(t,y)|_{y=0,1}=0$, which  implies $\partial_{y}\Bar{\Phi}(t,y)|_{y=0,1}=0.$  From \eqref{2.4}, we have:
\begin{equation}\label{2.5}
\partial_{t}{\partial_{y}\Psi}+k\partial_{y}\Psi=-C_0\partial_{y}^{3}\Bar{\Phi},
\end{equation}
then one has
\begin{equation}\label{2.6}
\partial_{y}^{3}\Bar{\Phi}(t,y)|_{y=0,1}=0.
\end{equation}

Repeating this  process,  we get the following compatible conditions
\begin{equation}\label{2.7}
 \partial_{y}^{2n-1}\Psi(t,y)|_{y=0,1}=0, \quad \partial_{y}^{2n-1}\Bar{\Phi}(t,y)|_{y=0,1}=0,\quad n\in\mathbb{N}.
 \end{equation}

\begin{proposition}
For $2\leq k\in \mathbb{N}$, and $(\bar{\Phi}(t,y),\Psi(t,y))$ is a pair of solutions to the system \eqref{2.4} and \eqref{2.7} with the initial data $(\bar{\Phi}_0(y),\Psi_0(y),\Psi_1(y))\in H^k$, then there exists a uniform constant $\alpha>0$, such that,
\begin{equation*}
\|\bar{\Phi}(t,\cdot)\|_{H^k([0,1])}
+
\|\partial_t\bar{\Phi}(t,\cdot)\|_{H^{k-1}([0,1])} \lesssim (\|\bar{\Phi}_0\|_{H^k}^2+ \|\Psi_0\|_{H^k}^2+ \|\Psi_1\|_{H^{k-1}}^2)e^{-\alpha t},
\end{equation*}
and
\begin{equation*}
\|\Psi(t,\cdot)\|_{H^{k-1}([0,1])}
+
\|\partial_t\Psi(t,\cdot)\|_{H^{k-2}([0,1])} \lesssim (\|\bar{\Phi}_0\|_{H^k}^2+ \|\Psi_0\|_{H^k}^2+ \|\Psi_1\|_{H^{k-1}}^2)e^{-\alpha t}.
\end{equation*}

\end{proposition}
\begin{proof}

By writing $\bar{\Phi}(t,y)=\Bar{\Phi}_{1}(y)\Bar{\Phi}_{2}(t)$, then from $\eqref{2.4}$, we have
\begin{equation}\label{2.8}
\begin{cases}
\partial^{2}_y\Bar{\Phi}_{1}(y)=\frac{\lambda}{C_0^{2}}\Bar{\Phi}_{1}(y),\\
 \partial^{2}_{t}\Bar{\Phi}_{2}(t)+\kappa\partial_{t}\Bar{\Phi}_{2}(t)=\lambda\Bar{\Phi}_{2}(t),\\
\partial_y \bar{\Phi}(t,y)|_{y=0,1}=0, \\ (\bar{\Phi},\Psi)|_{t=0}=(\bar{\Phi}_0(y),\Psi_0(y)),
\end{cases}
\end{equation}
with $\lambda$ being a constant.

Moreover, with   boundary conditions \eqref{1.8} and  \eqref{2.7}, we have
\begin{equation}\label{2.9}
\begin{cases}
\bar{\Phi}_0=\sum_{J=0}^\infty C_J \cos J\pi y, \\
\Psi_0=\sum_{J=0}^\infty D_J \cos J\pi y,
\end{cases} y\in [0,1],
\end{equation}
with
\begin{equation}\label{2.10}
C_J=2\int_0^1\bar{\Phi}_0 \cos (J\pi y) dy, \quad
D_J=2\int_0^1 \Psi_0 \cos (J\pi y) dy.
\end{equation}

For the case $\lambda\geq0$, by noting the boundary condition $\eqref{2.8}_3$,  it is obvious  that $\Bar{\Phi}_{1}(y)\equiv0$, which concludes that $\Bar{\Phi}(t, y)\equiv0$.

 When $\lambda<0$, from $\eqref{2.8}_1$, by a standard procedure, we get
 \begin{equation}\label{2.11}
 \Bar{\Phi}_{1}(y)=C_1^1cos(\sqrt{\frac{-\lambda}{C_0^{2}}}y)
+C_{2}^1sin(\sqrt{\frac{-\lambda}{C_0^{2}}}y).
\end{equation}
Recalling the boundary condition $\eqref{2.8}_3$, we get   $C_{2}^1=0$ and
\begin{equation}\label{2.12}
\lambda=-(C_0m\pi)^{2}, \quad \Bar{\Phi}_{1}(y)=C_{1}^1\cos(\sqrt{(m\pi)^{2}}y), \quad m\in{\mathbb{Z}}.
 \end{equation}

 From $\eqref{2.8}_2$ and \eqref{2.11}, we denote
\begin{equation*}
\gamma_{1,2}=\frac{-\kappa\pm\sqrt{\kappa^{2}-4(C_0m\pi)^{2}}}{2},
\end{equation*}
there are the following three different cases.

{\bf Case 1.}  $m^{2}<\frac{\kappa^{2}}{4C_0^{2}\pi^{2}}=\mathbb{K}_0^2$. In this case, there are only finite terms of $m$. We  get the solution of $\eqref{2.4}_1$ and \eqref{2.7}  as
 \begin{equation}\label{2.13}
\Bar{\Phi}^{(1)}(t,y)=\sum_{m< \mathbb{K}_0} (A_m e^{\gamma_1t}+B_m e^{\gamma_2t})\cos m\pi y, \quad m\in\mathbb{N}.
\end{equation}

Using the initial conditions \eqref{2.9} and $\eqref{2.4}_2$, then there holds
\begin{eqnarray}\label{2.14}
&&A_m=\frac{2}{\sqrt{\kappa^2-4C_0^2m^2\pi^2}}\big(\gamma_2\int_0^1\bar{\Phi}_0(y)\cos m\pi y dy+C_0\int_0^1\bar{\Psi}_0(y)\cos m\pi y dy\big) \\\nonumber
&&=\frac{1}{\sqrt{\kappa^2-4C_0^2m^2\pi^2}}\big( \gamma_2 C_m+  C_0 D_m \big),
\end{eqnarray}
and
\begin{eqnarray}\label{2.15}
&&B_m=\frac{-2}{\sqrt{\kappa^2-4C_0^2m^2\pi^2}}\big(\gamma_1\int_0^1\bar{\Phi}_0(y)\cos m\pi y dy+C_0\int_0^1\bar{\Psi}_0(y)\cos m\pi y dy\big) \\\nonumber
&&=\frac{-1}{\sqrt{\kappa^2-4C_0^2m^2\pi^2}}\big( \gamma_1 C_m+  C_0 D_m\big),
\end{eqnarray}
with $C_m$, $D_m$ are given by \eqref{2.10} and $m\in \mathbb{N}$.

Therefore, by \eqref{2.9}, \eqref{2.12}, \eqref{2.14} and \eqref{2.15}, we get
 \begin{eqnarray}\label{2.16}
\Bar{\Phi}^{(1)}(t,y)&=&\sum_{m< \mathbb{K}_0}\frac{1}{\sqrt{\kappa^2-4C_0^2m^2\pi^2}}\big(\gamma_2C_m+C_0D_m)(e^{2\gamma_1t}+e^{2\gamma_2t)} \\\nonumber
&&+\frac{1}{\sqrt{\kappa^2-4C_0^2m^2\pi^2}}\big( C_0D_m-\gamma_1C_m)e^{\gamma_2t}\cos m\pi y.
\end{eqnarray}

 Thus taking $
\alpha_0=\min\{\frac{\kappa-\sqrt{\kappa^{2}-4(C_0m\pi)^{2}}}{2}|_{m^{2}<\frac{\kappa^{2}}{4C_0^{2}\pi^{2}}}\}, $ we have
 \begin{eqnarray}\label{2.17}
\|\partial_y^k\Bar{\Phi}^{(1)}(t,y)\|_{L^2[0,1]}^2&\lesssim& e^{-2\alpha_0t}\sum_{m\leq \mathbb{K}_0}  (C_m^2+C_0^2D_m^2)m^{2k}\pi^{2k}\\\nonumber
  &\lesssim& e^{-2\alpha_0t}(\|\bar{\Phi}_0\|_{H^k}^2+C_0^2\|\Psi_0\|_{H^k}^2).
\end{eqnarray}

{\bf Case 2. $m^{2}=\frac{\kappa^{2}}{4C_0^{2}\pi^{2}}=\mathbb{K}_0^2$}. In this case, $\gamma_1=\gamma_2=-\frac{\kappa}2,$ we get the solution for  $\eqref{2.4}_1$, \eqref{2.7} and \eqref{2.9} as
 \begin{equation}\label{2.18}
\Bar{\Phi}^{(2)}(t,y)= (A_{\mathbb{K}_0}^{(1)}+B_{\mathbb{K}_0}^{(1)}t) e^{-\frac{\kappa}{2}t} \cos {\mathbb{K}_0}\pi y, \quad m\in\mathbb{N}.
\end{equation}
Similarly to \eqref{2.14}, we have
\begin{equation}\label{2.19}
 A_{\mathbb{K}_0}^{(1)}=C_{\mathbb{K}_0}, \quad  B_{\mathbb{K}_0}^{(1)}
 =\frac{\kappa}{2}C_{\mathbb{K}_0}-C_0D_{\mathbb{K}_0}.
\end{equation}

From \eqref{2.18}, \eqref{2.19} and \eqref{2.9}, for this case we have
 \begin{equation}\label{2.20}
\|\partial_y^k\Bar{\Phi}^{(2)}(t,y)\|_{L^2[0,1]}^2\lesssim e^{-\frac34\kappa t}(\|\bar{\Phi}_0\|_{H^k}^2+C_0^2\|\Psi_0\|_{H^k}^2) .
\end{equation}

{\bf Case 3. $m^{2}>\frac{\kappa^{2}}{4C_0^{2}\pi^{2}}=\mathbb{K}_0^2$}. In this case,  we have
 \begin{eqnarray}\label{2.21}
\Bar{\Phi}^{(3)}(t,y)&=&\sum_{m> \mathbb{K}_0}^\infty \big(A_m^{(2)}\cos\sqrt{\frac{4(C_0m\pi)^{2}-\kappa^{2}}{4}}t
\\\nonumber
&&+B_m^{(2)}   \sin\sqrt{\frac{4(C _0m\pi)^{2}-\kappa^{2}}{4}}t \big) e^{-\frac{ \kappa}{2}t} \cos m\pi y.
\end{eqnarray}

Again, by using  initial conditions \eqref{2.9} and \eqref{2.21},
 we get
 \begin{eqnarray}\label{2.22}
\Bar{\Phi}^{(3)}(t,y)&=& e^{-\frac{ \kappa}{2}t}\sum_{m>\mathbb{K}_0}^\infty \cos m\pi y\big(C_m \cos\sqrt{\frac{4(C_0m\pi)^{2}-\kappa^{2}}{4}}t +
\\\nonumber
&&+\frac{4}{\sqrt{4C_0^2m^2\pi^2-\kappa^2}}\big(\frac{\kappa}{2}C_m-C_0D_m\big)  \sin\sqrt{\frac{4(C _0m\pi)^{2}-\kappa^{2}}{4}}t \big).
\end{eqnarray}

Furthermore, we have
 \begin{eqnarray}\label{2.23}
\|\partial_y^k\Bar{\Phi}^{(3)}\|_{L^2[0,1]}^2&\lesssim & e^{-\kappa t}\sum_{m>\mathbb{K}_0}^\infty m^{2k}\pi^{2k}\bigg( C_m^2 \cos^2\sqrt{\frac{4(C_0m\pi)^{2}-\kappa^{2}}{4}}t +
\\\nonumber
&&+(C_m^2+C_0^2D_m^2)  \sin^2\sqrt{\frac{4(C _0m\pi)^{2}-\kappa^{2}}{4}}t \bigg)\\\nonumber
&\lesssim & e^{-\kappa t}(\|\bar{\Phi}_0\|_{H^k}^2+C_0^2\|\Psi_0\|_{H^k}^2).
\end{eqnarray}

In conclusion, taking $\alpha=\min\{\alpha_0, \frac{\kappa }4\},$ from \eqref{2.17}, \eqref{2.20} and \eqref{2.23}
we have
\begin{equation}\label{2.24}
\|\bar{\Phi}(t,\cdot)\|_{H^k([0,1])}\leq \sum_{i=1}^3 \|\bar{\Phi}^{(i)}(t,\cdot)\|_{H^k([0,1])} \lesssim (\|\bar{\Phi}_0\|_{H^k}^2+C_0^2\|\Psi_0\|_{H^k}^2)e^{-\alpha t},
\end{equation}
which also implies that
\begin{equation}\label{2.25}
\|\partial_t\Phi(t,\cdot)\|_{H^{k-1}([0,1])} \lesssim (\|\Psi_0\|_{H^k}^2+ \|\Psi_1\|_{H^k}^2)e^{-\alpha t}.
\end{equation}

Similarly, by using  $\eqref{2.4}_2$ and \eqref{2.8}, we have
\begin{equation}\label{2.26}
\|\Psi(t,\cdot)\|_{H^{k-1}([0,1])}
\ \mbox{and} \
\|\partial_t\Psi(t,\cdot)\|_{H^{k-2}([0,1])} \lesssim (\|\Psi_0\|_{H^k}^2+ \|\Psi_1\|_{H^k}^2)e^{-\alpha t}.
\end{equation}

\end{proof}

\subsection{Boundary conditions}
Take $\partial_y$ to  equation \eqref{1.6} and recall $\partial_{y}u_{1}|_{y=0,1}=0$, we have
\begin{equation}\label{2.27}
\big(\partial_{t}\partial_{y}\rho  +(u_{1}+\Psi)\partial_{x}\partial_{y}\rho+
\partial_{y}u_{2}\partial_{y}\rho\big)|_{y=0,1}=0.
\end{equation}

Since we have the initial condition  $\partial_{y}^{2n-1}\rho_0 (x,y)|_{y=0,1}=0,$
  we obtain
\begin{equation}\label{2.28}
\partial_{y}\rho(x,y)|_{y=0,1}=0.
\end{equation}

Writing $ \omega=\nabla^\bot \cdot   u $,  then the boundary condition \eqref{1.8} implies   $\omega|_{y=0,1}=0.$ Besides, from the incompressible constraint  $\nabla\cdot u=0$  we have
\begin{equation}\label{2.29}
\partial_{y}^{2}u_{2}|_{y=0,1}=0.
\end{equation}

From $\eqref{1.7}$, we have
\begin{multline}\label{2.30}
\partial_{t}\omega-\partial_{t}(\partial_{y}\Psi)+(u_{1}+\Psi)\partial_{x} \omega+u_{2}\partial_{y}(\omega-\partial_{y}\Psi)+\kappa(\omega-\partial_{y}\Psi)\\
=(\partial_{y}\rho+\partial_{y}\Bar{\Phi})(-\partial_{x}^{3}\rho-\partial_{x}\partial_{y}^{2}\rho)
+(\partial_{x}\rho+C_0)(\partial_{x}^{2}\partial_{y}\rho+\partial_{y}^{3}\rho+\partial_{y}^{3}\Bar{\Phi}).
\end{multline}

By using  \eqref{1.6}, \eqref{2.7} and \eqref{2.30}, there holds
\begin{equation}\label{2.31}
\partial_{y}^{3}\rho|_{y=0,1}=0.
\end{equation}

Moreover, from \eqref{1.6} we also have
\begin{multline}\label{2.32}
\partial_{t}\partial_{y}^{3}\rho+\partial_{y}^{3}u_2\partial_{y}\bar{\Phi}+u_2\partial_y^4\bar{\Phi}+3\partial_{y}^2u_2\partial_y^2\bar{\Phi}+3\partial_yu_2\partial_y^3\bar{\Phi}+C_0\partial_y^3u_1+(\partial_y^3u\cdot\nabla)\rho\\
+(u\cdot\nabla)\partial_y^3\rho+3(\partial_y^2u\cdot\nabla)\partial_y\rho+3(\partial_yu\cdot\nabla)\partial_y^2\rho+\partial_y^3\Psi\partial_x\rho+\Psi\partial_x\partial_y^3\rho\\
+3\partial_y\Psi\partial_x\partial_y^2\rho+3\partial_y^2\Psi\partial_{xy}\rho=0.
\end{multline}

Combining \eqref{2.7}, \eqref{2.30} with \eqref{2.28}-\eqref{2.32}, we get
\begin{equation}\label{2.33}
\partial_{y}^{3}u_{1}|_{y=0,1}=0.
\end{equation}

Repeating the iterative process above, we get  boundary  compatible conditions as
\begin{equation}\label{2.34}
\partial_{y}^{2n}u_{2}|_{y=0,1}=0,\quad \partial_{y}^{2n-1}u_{1}|_{y=0,1}=0,\quad \partial_{y}^{2n-1}\rho|_{y=0,1}=0, \quad n\in\mathbb{N}^+.
\end{equation}
\subsection{Notations}
In order to solve our problem   in a certain Sobolev space, recalling the boundary conditions \eqref{2.34}, we  introduce  following function spaces:
\begin{equation}\label{2.35}
\begin{cases}
X^{k}(\Omega)=\{f\in H^{k}(\Omega):\partial^{2n}_{y}f|_{\partial\Omega}=0,n\in \mathbb{N}^+\}
,\\
Y^{k}(\Omega)=\{f\in H^{k}(\Omega):\partial^{2n-1}_{y}f|_{\partial\Omega}=0,n\in \mathbb{N}^+\}.
\end{cases}
\end{equation}

To our problem, we  define the following functional space as:
\begin{equation}\label{2.36}
\aleph^k(\Omega):=\{\emph{v}\in H^{k}(\Omega),\emph{v}=(v_{1},v_{2})\in Y^{k}(\Omega)\times X^{k}(\Omega)\}.
\end{equation}

For convenience, in this paper, we use $L^{2}$, $\dot{H}^{k}$ and $H^{k}$ to stand for $L^{2}(\Omega)$, $\dot{H}^{k}(\Omega)$ and $H^{k}(\Omega)$ respectively. We also use $<,> $ as the scalar product  in $L^2$.

\subsection{An orthonormal basis for $X^{k}(\Omega),Y^{k}(\Omega)$.}
Let us start by defining
\begin{equation}\label{2.37}
a_{q}(y):=
\begin{cases}
cos(qy\frac{\pi}{2}) \quad q\quad odd,\\
sin(qy\frac{\pi}{2}) \quad q\quad even,
\end{cases}
with \ y\in [0,1], \quad  q\in  \mathbb{ N}^+\cup\{0\},
\end{equation}
and
\begin{equation}\label{2.38}
b_{q}(y):=
\begin{cases}
sin(qy\frac{\pi}{2}) \quad q\quad odd,\\
cos(qy\frac{\pi}{2}) \quad q\quad even,
\end{cases}
 with \ y\in [0,1], \quad q\in \mathbb{ N}^+\cup\{0\},
\end{equation}
where $(\{a_{q}\}_{q\in \mathbb{N}},\{b_{q}\}_{q\in \mathbb{N}})$ is the  orthonormal basis for $L^{2}([0,1])$.
\begin{remark}
Actually, the basis given above  is chosen based on the boundary conditions and the Fourier series. More precisely,   for $f\in L^{2}(\Omega)$, we have the $L^{2}(\Omega)$-convergence series  given by:
\begin{equation*}
  f(x,y)= \sum_{q\in \mathbb{ N}\cup { 0} } \mathcal{F}_{a}[f](x,q)a_{q}(y) \quad where\quad  \mathcal{F}_a[f](x,q):=\int_0^1 f (x ,y )\overline{a_q(y)}dy
\end{equation*}
or\\
\begin{equation*}
  f(x,y)=\sum_{q\in \mathbb{N}\cup\{0\}} \mathcal{F}_{b}[f](x,q)b_{q}(y) \quad where\quad  \mathcal{F}_{b}[f](x,q):=\int_0^1f(x,y)\overline{b_{q}(y)}dy,
\end{equation*}
\end{remark}
with $\overline{a_q(y)}$ and $\overline{b_{q}(y)}$ are the conjugate of  $a_q(y)$ and $b_q(y)$ respectively.

\section{The energy estimates}
For simplicity,  we denote
\begin{equation}\label{3.1}
E_{k}^2(t):=\frac{1}{2}\{\|u\|_{H^{k}}^{2}+\|\partial_{t}u\|_{H^{k}}^{2}+\|\rho\|_{H^{k+1}}^{2}+\|\partial_{t}\rho\|_{H^{k+1}}^{2}\},
\end{equation}
and the following  weighted energy, which   plays a crucial role in our proof,
\begin{equation}\label{3.2}
\Gamma_{k}^2(t):=\frac{1}{2}\{\int_{\Omega}(1+\partial_{x}\rho)|\nabla^ku|^{2}+(1+\partial_{x}\rho)|\nabla^{k}\nabla^{\bot}\rho|^{2}dxdy\}.
\end{equation}

 \subsection{Local well-posedness}
 To prove our main theorem, we first claim  the local well-posedness for the system \eqref{1.6}-\eqref{1.8}, which can be verified by a standard Galerkin procedure. Here, we omit the proof for simplicity, and give the results as follows:
 \begin{proposition}
For the system  \eqref{1.6}-\eqref{1.8}, with  the initial data $(u_0,\rho_0)\in \aleph^k(\Omega)\times  Y^{k+1}(\Omega)$ and $(\Psi_0,\bar{\Phi}_0)\in Y^{k+3}$, then there exists a positive constant $T$, such that the system  \eqref{1.6}-\eqref{1.8} admits a unique pair of solutions $(u,\rho)$ on $[0,T]$. Moreover, for  $k\geq 3$ there holds
 \begin{equation}\label{3.3}
 \sup_{0\leq t\leq T}E_k^2(t)\leq M E_k^2(0),
 \end{equation}
 with $M$ being a constant.
 \end{proposition}

\subsection{Energy Estimates}

\begin{proposition}(Energy Estimates) Under the condition of Proposition 3.1, we assume that the initial data $(u_0,\rho_0)\in \aleph^k(\Omega)\times  Y^{k+1}(\Omega)$ is small enough,
 such that
\begin{equation}\label{3.4}
\|\partial_x \rho(t,\cdot)\|_{H^2(\Omega)} <\frac14, \quad t\in [0,T],
\end{equation}
then  we have  the following energy  estimates for the system \eqref{1.6}-\eqref{1.9}:
\begin{equation}\label{3.5}
\frac{d}{dt} E_k^2 \lesssim (E_k^2+\Gamma_{k+1}^2)(E_4 +\|\bar{\Phi}\|_{H^{k+2}}+\|\Psi\|_{H^{k+2}}+\|\partial_t\Psi\|_{H^{k+1}}).
\end{equation}

\end{proposition}
\begin{proof}
Through the standard energy estimates, noting $u=(u_1,u_2)^T$,  we have
\begin{multline}\label{3.6}
\frac{1}{2}\frac{d}{dt}\|u\|_{L^{2}}^{2}=<u,\partial_{t}u>
=-\kappa\|u\|_{L^2}^2-<u_1,u_{2}\partial_{y}\Psi+\partial_{x}\rho\partial_{y}^{2}\Bar{\Phi}>\\
+<u,(\nabla^{\bot}\rho\cdot\nabla)\nabla^{\bot}\rho+\partial_{y}\Bar{\Phi}(\partial_{xy}\rho,-\partial_{x}^{2}\rho)^T-C_0(\partial_{y}^{2}\rho,-\partial_{xy}\rho)^T>.
\end{multline}

Similarly, by using $\nabla\cdot u=0$  and the boundary condition we have
\begin{eqnarray}\label{3.7}
&&\frac{1}{2}\frac{d}{dt}\|u\|_{\dot{H^{k}}}^{2}=<\nabla^ku,\nabla^k\partial_{t}u>\\\nonumber
&&=<\nabla^ku,\nabla^k[-(u\cdot\nabla)u-\kappa{u}-\Psi\partial_{x}u]>-<\nabla^ku_1,\nabla^k(u_{2}\partial_{y}\Psi+\partial_{x}\rho\partial_{y}^{2}\Bar{\Phi})>\\\nonumber
&&+<\nabla^ku,\nabla^k[(\nabla^{\bot}\rho\cdot\nabla)\nabla^{\bot}\rho+\partial_{y}\Bar{\Phi}(\partial_{xy}\rho,-\partial_{x}^{2}\rho)^T-C_0(\partial_{y}^{2}\rho,-\partial_{xy}\rho)^T]>,
\end{eqnarray}
and
\begin{multline}\label{3.8}
\frac{1}{2}\frac{d}{dt}\|\partial_tu\|_{L^{2}}^{2} =<\partial_{t}u,\partial_{t}[-(u\cdot\nabla)u-\kappa u-\Psi\partial_{x}u
+(\nabla^{\bot}\rho\cdot\nabla)\nabla^{\bot}\rho]>\\
+\partial_{y}\Bar{\Phi}(\partial_{xy}\rho,- \partial_{x}^{2}\rho)^T
-C_0(\partial_{y}^{2}\rho,- \partial_{xy}\rho)^T]>
-<\partial_t u_1,\partial_t[u_2 \partial_y \Psi+\partial_x\rho \partial_y^2\Bar{\Phi}]>,
\end{multline}
as well as
\begin{eqnarray}\label{3.9}
\frac{1}{2}\frac{d}{dt}\|\partial_{t}u\|_{\dot{H^{k}}}^{2}
&=&<\nabla^k\partial_{t}u,\nabla^k\partial_{t}[-(u\cdot\nabla)u-\Psi\partial_{x}u
-\kappa u
\\\nonumber
&&+(\nabla^{\bot}\rho\cdot\nabla)\nabla^{\bot}\rho+\partial_{y}\Bar{\Phi}(\partial_{xy}\rho,-\partial_{x}^{2}\rho)^T
-C_0(\partial_{y}^{2}\rho,-\partial_{xy}\rho)^T]>
\\\nonumber
&&-<\partial_t \partial^ku_1,\partial_t \partial^k[u_2 \partial_y \Psi+\partial_x\rho \partial_y^2\Bar{\Phi}]>.
\end{eqnarray}

Applying $\nabla^\bot$ to \eqref{1.6} and then by a similar procedure, we have
\begin{eqnarray}\label{3.10}
&&\frac12\frac{d}{dt}\|\nabla^{\bot}\rho\|_{\dot{H^{k}}}^{2}=<\nabla^k\nabla^{\bot}\rho,\nabla^k\partial_{t}\nabla^{\bot}\rho>\\\nonumber
&&=<\nabla^k\nabla^{\bot}\rho,\nabla^k\nabla^{\bot}[-u_{2}\partial_{y}\Bar{\Phi}-C_0u_{1}-(u\cdot\nabla)\rho-\Psi\partial_{x}\rho]>,
\end{eqnarray}
and
\begin{eqnarray}\label{3.11}
&&\frac12\frac{d}{dt}\|\nabla^{\bot}\partial_{t}\rho\|_{\dot{H^{k}}}^{2}=<\nabla^k\nabla^{\bot}\partial_{t}\rho,\nabla^k\partial_{t}^{2}\nabla^{\bot}\rho>\\\nonumber
&&=<\nabla^k\nabla^{\bot}\partial_{t}\rho,\nabla^k\nabla^{\bot}\partial_{t}[-u_{2}\partial_{y}\Bar{\Phi}-C_0u_{1}-(u\cdot\nabla)\rho-\Psi\partial_{x}\rho]>.
\end{eqnarray}

From \eqref{3.6}-\eqref{3.11}, we get
\begin{equation}\label{3.12}
\frac{1}{2}
\frac{d}{dt}
(\|u\|_{\dot{H}^k}^{2}+\|\partial_{t}u\|_{\dot{H^{k}}}^{2}+\|\nabla^{\bot}\rho\|_{\dot{H^{k}}}^{2}+\|\nabla^{\bot}\partial_{t}\rho\|_{\dot{H^{k}}}^{2})+\kappa(\|u\|_{\dot{H^{k}}}+\|\partial_{t}u\|_{\dot{H^{k}}})\triangleq \sum_{i=1}^3G_i,
\end{equation}
with
\begin{eqnarray}
&&G_1=-<u_1,u_{2}\partial_{y}\Psi+\partial_{x}\rho\partial_{y}^{2}\Bar{\Phi}>+<u,\partial_{y}\Bar{\Phi}(\partial_{xy}\rho,-\partial_{x}^{2}\rho)^T>\\\nonumber
&&+<u,(\nabla^{\bot}\rho\cdot\nabla)\nabla^{\bot}\rho>-<\partial_t u_1,\partial_t[u_2 \partial_y\Psi+\partial_x\rho\partial_y^2\Bar{\Phi}]>\\\nonumber
&&+<\partial_{t}u,\partial_{t}[\partial_{y}\Bar{\Phi}(\partial_{xy}\rho,-\partial_{x}^{2}\rho)^T]>-<\partial_{t}u,\partial_{t}(\Psi\partial_{x}u)>\\\nonumber
&&-<\partial_{t}u,\partial_{t}[(u\cdot\nabla)u]>+<\partial_{t}u,\partial_{t}[(\nabla^{\bot}\rho\cdot\nabla)\nabla^{\bot}\rho]>\\\nonumber
&&-<\nabla^{\bot}\rho,\nabla^{\bot}(\Psi\partial_{x}\rho)>-<\nabla^{\bot}\rho,\nabla^{\bot}[(u\cdot\nabla)\rho]>\\\nonumber
&&-<\nabla^{\bot}\rho,\nabla^{\bot}(u_{2}\partial_{y}\Bar{\Phi})>+<\nabla^{\bot}\partial_{t}\rho,\nabla^{\bot}\partial_{t}[(u\cdot\nabla)\rho]>\\\nonumber
&&-<\nabla^{\bot}\partial_{t}\rho,\nabla^{\bot}\partial_{t}(u_{2}\partial_{y}\Bar{\Phi})>-<\nabla^{\bot}\partial_{t}\rho,\nabla^{\bot}\partial_{t}(\Psi\partial_{x}\rho)>,
\end{eqnarray}
and
\begin{eqnarray}
&&G_2=-<\nabla^{k}u,\nabla^{k}[(u\cdot\nabla)u]>-<\nabla^{k}u_1,\nabla^{k}[u_2 \partial_y \Psi+\partial_x\rho \partial_y^2\Bar{\Phi}]>\\\nonumber
&&-<\nabla^{k}u,\nabla^{k}[\Psi\partial_{x}u]>-<\nabla^{k}\partial_tu_1,\nabla^{k}\partial_t[u_2\partial_y\Psi+\partial_x\rho\partial_y^2\Bar{\Phi}]>\\\nonumber
&&-<\nabla^{k}\partial_{t}u,\nabla^{k}\partial_{t}(\Psi\partial_{x}u)>-<\nabla^{k}\partial_{t}u,\nabla^{k}\partial_{t}[(u\cdot\nabla)u]>\\\nonumber
&&-<\nabla^{k}\nabla^{\bot}\rho,\nabla^{k}\nabla^{\bot}(\Psi\partial_{x}\rho)>-<\nabla^{k}\nabla^{\bot}\partial_{t}\rho,\nabla^{k}\nabla^{\bot}\partial_{t}(\Psi\partial_{x}\rho)>\\\nonumber
\end{eqnarray}
as well as
\begin{eqnarray}\label{3.15}
&& \\\nonumber
&&G_3=<\nabla^{k}u,\nabla^{k}[(\nabla^{\bot}\rho\cdot\nabla)\nabla^{\bot}\rho]>-<\nabla^{k}\nabla^{\bot}\rho,\nabla^{k}\nabla^{\bot}[(u\cdot\nabla)\rho]>\\\nonumber
&&+<\nabla^{k}\partial_{t}u,\nabla^{k}\partial_{t}[(\nabla^{\bot}\rho\cdot\nabla)\nabla^{\bot}\rho]>-<\nabla^{k}\nabla^{\bot}\partial_{t}\rho,\nabla^{k}\nabla^{\bot}\partial_{t}[(u\cdot\nabla)\rho]>\\\nonumber
&&+<\nabla^{k}u,\nabla^{k}(\partial_{y}\Bar{\Phi}\partial_{xy}\rho,-\partial_{y}\Bar{\Phi}\partial_{x}^{2}\rho)^T>-<\nabla^{k}\nabla^{\bot}\rho,\nabla^{k}\nabla^{\bot}(u_{2}\partial_{y}\Bar{\Phi})>\\\nonumber
&&+<\nabla^{k}\partial_{t}u,\nabla^{k}\partial_{t}(\partial_{y}\Bar{\Phi}\partial_{xy}\rho,-\partial_{y}\Bar{\Phi}\partial_{x}^{2}\rho)^T>-<\nabla^{k}\nabla^{\bot}\partial_{t}\rho,\nabla^{k}\nabla^{\bot}\partial_{t}(u_{2}\partial_{y}\Bar{\Phi})>\\\nonumber
&&-<\nabla^{k}u,C_0\nabla^{k}(\partial_{y}^{2}\rho,-\partial_{xy}\rho)^T>-<\nabla^{k}\nabla^{\bot}\rho,\nabla^{k}\nabla^{\bot}(C_0u_{1})>\\\nonumber
&&-<\nabla^{k}\partial_{t}u,C_0\nabla^{k}\partial_{t}(\partial_{y}^{2}\rho,-\partial_{xy}\rho)^T>-<\nabla^{k}\nabla^{\bot}\partial_{t}\rho,\nabla^{k}
\nabla^{\bot}\partial_{t}(C_0u_{1})>\\\nonumber
&&=\sum_{i=1}^6 I_i.
\end{eqnarray}

During the calculation given above, we used the incompressible condition and get the cancellation for some quadratic terms as
\begin{eqnarray}\label{3.16}
&& \\\nonumber
&&-<u_{1},C_0\partial_{y}^{2}\rho>+<u_{2},C_0\partial_{xy}\rho>-<\nabla^{\bot}\rho,\nabla^{\bot}(C_0u_{1})>
\\\nonumber
&&= -<u_{1},C_0\partial_{y}^{2}\rho>-<\partial_y u_{2},C_0\partial_{x}\rho>-<\partial_y\rho,C_0\partial_y u_{1}>-<\partial_x\rho,C_0\partial_x u_{1}>\\\nonumber
&&=-<u_{1},C_0\partial_{y}^{2}\rho>+<\partial_xu_{1},C_0\partial_{x}\rho>+<u_{1},C_0\partial_{y}^{2}\rho>-<\partial_x\rho,C_0\partial_x u_{1}>\\\nonumber
&&=0,
\end{eqnarray}
and
\begin{equation}\label{3.17}
-<\partial_{t}u_{1},C_0\partial_{y}^{2}\partial_{t}\rho>-<\partial_{t}u_{2},C_0\partial_{xy}\partial_{t}\rho>-<\nabla^{\bot}\partial_{t}\rho,\nabla^{\bot}
\partial_{t}(C_0u_{1})>
=0.
\end{equation}

Noting that $k\geq 7$, therefore $G_1$ is a lower order term. It is easy to get that
\begin{equation}\label{3.18}
G_1\lesssim E_4^2(\|u\|_{H^3}+\|\nabla \rho\|_{H^3}+\|\Psi\|_{H^3}+\|\partial_{t}\Psi\|_{H^2}+\|\bar{\Phi}\|_{H^3}).
\end{equation}

By using Lemma 2.1, for the term $G_2$, we have
\begin{eqnarray}\label{3.19}
G_2&\lesssim& E_k^2(E_4+\|\bar{\Phi}\|_{H^{3}}+\|\Psi\|_{H^{3}}+\|\partial_{t}\Psi\|_{H^{2}})\\\nonumber
&&+(\|\bar{\Phi}\|_{H^{k+2}} +\|\Psi\|_{H^{k+1}} +\|\partial_{t}\Psi\|_{H^{k+1}} )E_kE_4\\\nonumber
&&+<\partial^{k}\partial_{t}u,\partial_{t}\Psi\partial^{k}\partial_{x}u>+<\partial^{k}\nabla^{\bot}\partial_{t}\rho,\partial_{t}\Psi
\partial^{k}\nabla^{\bot}\partial_{x}\rho>\\\nonumber
&&+<\partial^{k}\partial_{t}u,(\partial_{t}u\cdot\nabla)\partial^{k}u>.
\end{eqnarray}
Recalling \eqref{3.4}, we have
\begin{eqnarray}\label{3.20}
&&|<\partial^{k}\partial_{t}u,\partial_{t}\Psi\partial^{k}\partial_{x}u>+<\partial^{k}\nabla^{\bot}\partial_{t}\rho,\partial_{t}\Psi\partial^{k}\nabla^{\bot}\partial_{x}\rho>\\\nonumber
&&+<\partial^{k}\partial_{t}u,(\partial_{t}u\cdot\nabla)\partial^{k}u>|\\\nonumber
&&\lesssim\frac{1}{(1-\|\partial_{x}\rho\|_{L^{\infty}})^\frac{1}{2}}\|\partial_{t}u\|_{H^{k}}\|\partial_{t}\Psi\|_{L^{\infty}}[\int_{\Omega}(1+\partial_{x}\rho)|\partial^{k+1}u|^2)dxdy]^{\frac{1}{2}}\\\nonumber
&&+\frac{1}{(1-\|\partial_{x}\rho\|_{L^{\infty}})^\frac{1}{2}}\|\partial_{t}\rho\|_{H^{k+1}}\|\partial_{t}\Psi\|_{L^{\infty}}
[\int_{\Omega}(1+\partial_{x}\rho)|\nabla^{k+1}\nabla^{\bot}\rho|^2)dxdy]^{\frac{1}{2}}\\\nonumber
&&+\frac{1}{(1-\|\partial_{x}\rho\|_{L^{\infty}})^\frac{1}{2}}\|\partial_{t}u\|_{H^{k}}\|\partial_{t}u\|_{L^{\infty}}
[\int_{\Omega}(1+\partial_{x}\rho)|\nabla^{k+1}u|^2)dxdy]^{\frac{1}{2}}\\\nonumber
&&\lesssim\frac{\Gamma_{k+1}}{(1-\|\partial_{x}\rho\|_{L^{\infty}})^\frac{1}{2}}(\|\partial_{t}u\|_{H^{k}}+\|\partial_{t}\rho\|_{H^{k+1}})(\|\partial_{t}\Psi\|_{H^1}+\|\partial_{t}u\|_{H^2}).
\end{eqnarray}

From \eqref{3.19} and \eqref{3.20}, we have
\begin{equation}\label{3.21}
G_2\lesssim  (\Gamma_{k+1}^2+E_k^2)(E_4+\|\bar{\Phi}\|_{H^{k+2}}+\|\Psi\|_{H^{k+1}}+\|\partial_{t}\Psi\|_{H^{k+1}}).
\end{equation}

Next, for the term $G_3$,  we shall estimate $I_i, (1\leq i\leq 6)$ term by term.

For the term $I_{1}$, by using Lemma 2.1 and the incompressible condition,  we have:
\begin{eqnarray}\label{3.22}
&& \\\nonumber
&&\quad I_{1}\lesssim (\|u\|_{H^k}^2+\|\rho\|_{H^{k+1}}^2)(\|\nabla\rho\|_{H^3}+\|u\|_{H^3})\\\nonumber
&&+<\nabla^ku_1,\partial_y\rho\nabla^k\partial_{xy}\rho-\partial_x\rho\nabla^k\partial_{y}^{2}\rho>
-<\nabla^{k}u_{2},\partial_{y}\rho\nabla^{k}\partial_{x}^{2}\rho-\partial_{x}\rho\nabla^k\partial_{xy}\rho>\\\nonumber
&&-<\nabla^k\partial_{y}\rho,\nabla^k\partial_{y}u_1\partial_{x}\rho+\nabla^{k}\partial_{y}u_2\partial_{y}\rho>-<\nabla^{k}\partial_{x}\rho,\nabla^{k}\partial_{x}u_1
\partial_{x}\rho+\nabla^k\partial_{x}u_{2}\partial_{y}\rho>.
\end{eqnarray}
Due to $\nabla\cdot u=0$ and Lemma 2.1, we have
\begin{eqnarray}\label{3.23}
&&<\nabla^ku_1,\partial_y\rho\nabla^k\partial_{xy}\rho-\partial_x\rho\partial^k\partial_{y}^{2}\rho>-<\nabla^{k}u_{2},\partial_{y}\rho\nabla^{k}\partial_{x}^{2}\rho-\partial_{x}\rho\nabla^k\partial_{xy}\rho>\\\nonumber
&&-<\partial^{k}\nabla^{\bot}\rho,(\partial^{k}\nabla^{\bot}u\cdot\nabla)\rho>\\\nonumber
&&=-\int_{\Omega}\nabla^{k}\partial_{x}u_{1}\partial_{y}\rho\nabla^{k}\partial_{y}\rho\,dxdy-\int_{\Omega}\nabla^k\partial_{y}\rho\nabla^{k}\partial_{y}u_2\partial_{y}\rho\,dxdy\\\nonumber
&&+\int_{\Omega}\nabla^{k}\partial_{y}u_{1}\partial_{x}\rho\nabla^{k}\partial_{y}\rho\,dxdy-\int_{\Omega}\nabla^k\partial_{y}\rho\nabla^k\partial_{y}u_1\partial_{x}\rho\,dxdy\\\nonumber
&&+\int_{\Omega}\nabla^{k}\partial_{x}u_{2}\partial_{y}\rho\nabla^{k}\partial_{x}\rho\,dxdy-\int_{\Omega}\nabla^{k}\partial_{x}\rho\nabla^{k}\partial_{y}u_2\partial_{y}\rho\,dxdy\\\nonumber
&&-\int_{\Omega}\nabla^{k}\partial_{y}u_{2}\partial_{x}\rho\nabla^{k}\partial_{x}\rho\,dxdy-\int_{\Omega}\nabla^{k}\partial_{x}\rho\nabla^{k}\partial_{x}u_1\partial_{x}\rho\,dxdy\\\nonumber
&&-\int_{\Omega}\nabla^{k}u_{1}\partial_{xy}\rho\nabla^{k}\partial_{y}\rho\,dxdy+\int_{\Omega}\nabla^{k}u_{1}\partial_{xy}\rho\nabla^{k}\partial_{y}\rho\,dxdy\\\nonumber
&&+\int_{\Omega}\nabla^{k}u_{2}\partial_{xy}\rho\nabla^{k}\partial_{x}\rho\,dxdy-\int_{\Omega}\nabla^{k}u_{2}\partial_{xy}\rho\nabla^{k}\partial_{x}\rho\,dxdy=0.
\end{eqnarray}

Combining \eqref{3.22} with \eqref{3.23}, we get
\begin{equation}\label{3.24}
I_{1}\lesssim E_k^2E_4.
\end{equation}

As to the term $I_2$, we have
\begin{eqnarray}\label{3.25}
&&I_{2}=<\nabla^k\partial_{t}u,\nabla^k[(\partial_{t}\nabla^{\bot}\rho\cdot\nabla)\nabla^{\bot}\rho+(\nabla^{\bot}\rho\cdot\nabla)\nabla^{\bot}\partial_{t}\rho]>\\\nonumber
&&-<\nabla^k\nabla^{\bot}\partial_{t}\rho,\nabla^k[(\partial_{t}\nabla^{\bot}u\cdot\nabla)\rho+(\partial_{t}u\cdot\nabla)\nabla^{\bot}\rho\\\nonumber
&&+(\nabla^{\bot}u\cdot\nabla)\partial_{t}\rho+(u\cdot\nabla)\partial_{t}\nabla^{\bot}\rho]>\\\nonumber
&&\lesssim E_k^2E_4+<\nabla^k\partial_{t}u,(\partial_{t}\nabla^{\bot}\rho\cdot\nabla)\nabla^k\nabla^{\bot}\rho>\\\nonumber
&&-<\nabla^k\nabla^{\bot}\partial_{t}\rho,(\partial_{t}u\cdot\nabla)\nabla^k\nabla^{\bot}\rho+<\nabla^{k}\partial_{t}u,(\nabla^{\bot}\rho\cdot\nabla)\nabla^k\nabla^{\bot}\partial_{t}\rho>\\\nonumber
&&-<\nabla^{k}\nabla^{\bot}\partial_{t}\rho,(\nabla^k\partial_{t}\nabla^{\bot}u\cdot\nabla)\rho-<\nabla^k\nabla^{\bot}\partial_{t}
\rho,(\nabla^{k}\nabla^{\bot}u\cdot\nabla)\partial_{t}\rho>.
\end{eqnarray}

Similar to \eqref{3.16} and \eqref{3.23}, we have:
\begin{eqnarray}\label{3.26}
&&<\nabla^{k}\partial_{t}u,(\nabla^{\bot}\rho\cdot\nabla)\nabla^k\nabla^{\bot}\partial_{t}\rho>-<\nabla^{k}\nabla^{\bot}\partial_{t}\rho,(\nabla^k\partial_{t}\nabla^{\bot}u\cdot\nabla)\rho>\\\nonumber
&&=-\int_{\Omega}\nabla^{k}\partial_{t}\partial_{x}u_1\partial_{y}\rho\nabla^{k}\partial_{ty}\rho\,dxdy-\int_{\Omega}\nabla^{k}\partial_{t}\partial_{y}\rho\nabla^k\partial_{t}\partial_{y}u_{2}\partial_{y}\rho\,dxdy\\\nonumber
&&+\int_{\Omega}\nabla^{k}\partial_{t}\partial_{y}u_1\partial_{x}\rho\nabla^{k}\partial_{ty}\rho\,dxdy-\int_{\Omega}\nabla^k\partial_{t}\partial_{y}\rho\nabla^k\partial_{t}\partial_{y}u_{1}\partial_{x}\rho\,dxdy\\\nonumber
&&+\int_{\Omega}\nabla^{k}\partial_{t}\partial_{x}u_2\partial_{y}\rho\nabla^{k}\partial_{tx}\rho\,dxdy-\int_{\Omega}\nabla^k\partial_{t}\partial_{x}\rho\nabla^k\partial_{t}\partial_{x}u_{2}\partial_{y}\rho\,dxdy\\\nonumber
&&-\int_{\Omega}\nabla^{k}\partial_{t}\partial_{y}u_2\partial_{x}\rho\nabla^{k}\partial_{tx}\rho\,dxdy-\int_{\Omega}\nabla^k\partial_{t}\partial_{x}\rho\nabla^k\partial_{t}\partial_{x}u_{1}\partial_{x}\rho\,dxdy\\\nonumber
&&-\int_{\Omega}\nabla^{k}\partial_{t}u_1\partial_{xy}\rho\nabla^{k}\partial_{ty}\rho\,dxdy+\int_{\Omega}\nabla^{k}\partial_{t}u_1\partial_{xy}\rho\nabla^{k}\partial_{ty}\rho\,dxdy\\\nonumber
&&+\int_{\Omega}\nabla^{k}\partial_{t}u_2\partial_{xy}\rho\nabla^{k}\partial_{tx}\rho\,dxdy-\int_{\Omega}\nabla^{k}\partial_{t}u_2\partial_{xy}\rho\nabla^{k}\partial_{tx}\rho\,dxdy=0,
\end{eqnarray}
and
\begin{eqnarray}\label{3.27}
&&<\nabla^k\partial_{t}u,(\partial_{t}\nabla^{\bot}\rho\cdot\nabla)\nabla^k\nabla^{\bot}\rho>
+<\nabla^k\nabla^{\bot}\partial_{t}\rho,(\partial_{t}u\cdot\nabla)\nabla^k\nabla^{\bot}\rho>\\\nonumber
&&+<\nabla^k\nabla^{\bot}\partial_{t}\rho,(\nabla^{k}\nabla^{\bot}u\cdot\nabla)\partial_{t}\rho>\\\nonumber
&&\lesssim\frac{\|\partial_{t}u\|_{H^{k}}\|\nabla\partial_{t}\rho\|_{L^{\infty}}}{(1-\|\partial_{x}\rho\|_{L^{\infty}})^{\frac{1}{2}}}(\int_{\Omega}(1+\partial_{x}\rho)|\nabla^{k+1}\nabla^{\bot}\rho|^{2}dxdy)^{\frac{1}{2}}\\\nonumber
&&+\frac{\|\partial_{t}\rho\|_{H^{k+1}}\|\partial_{t}u\|_{L^{\infty}}}{(1-\|\partial_{x}\rho\|_{L^{\infty}})^{\frac{1}{2}}}(\int_{\Omega}(1+\partial_{x}\rho)|\nabla^{k+1}\nabla^{\bot}\rho|^{2}dxdy)^{\frac{1}{2}}\\\nonumber
&&+\frac{\|\partial_{t}\rho\|_{H^{k+1}}\|\partial_{t}\nabla\rho\|_{L^{\infty}}}{(1-\|\partial_{x}\rho\|_{L^{\infty}})^{\frac{1}{2}}}(\int_{\Omega}(1+\partial_{x}\rho)|\partial^{k+1}u|^{2}dxdy)^{\frac{1}{2}}.
\end{eqnarray}

From \eqref{3.25}-\eqref{3.27},
we have:
\begin{equation}\label{3.28}
I_{2}\lesssim E_k^2E_4+\frac{E_kE_4}{(1-\|\partial_{x}\rho\|_{L^{\infty}})^{\frac{1}{2}}}\Gamma_{k+1}\lesssim(E_k^2+\Gamma_{k+1}^2)E_4.
\end{equation}

For $I_{3}$, we have
\begin{eqnarray}\label{3.29}
&&I_{3}\lesssim  E_k^2 \|\bar{\Phi}\|_{H^{3}}+E_k\|\bar{\Phi}\|_{H^{k+2}} E_4\\\nonumber
&&+<\nabla^{k}u,(\partial_{y}\Phi\nabla^{k}\partial_{xy}\rho,-\partial_{y}\Phi\nabla^{k}\partial_{x}^{2}\rho)>
-<\nabla^{k}\nabla^{\bot}\rho,\nabla^{k}\nabla^{\bot}u_2\partial_{y}\Phi>.
\end{eqnarray}
Moreover, since  $\Bar{\Phi}$ is independent of  $x$,  we get
\begin{eqnarray}\label{3.30}
&&<\nabla^{k}u,(\partial_{y}\Phi\nabla^{k}\partial_{xy}\rho,-\partial_{y}\Phi\nabla^{k}\partial_{x}^{2}\rho)^T>-<\nabla^{k}\nabla^{\bot}\rho,\nabla^{k}\nabla^{\bot}u_2\partial_{y}\Phi>\\\nonumber
&&=\int_{\Omega}\nabla^ku_{1}\partial_{y}\Bar{\Phi}\nabla^k\partial_{xy}\rho\,dxdy-\int_{\Omega}\nabla^ku_{2}\partial_{y}\Bar{\Phi}\nabla^k\partial_{x}^{2}\rho\,dxdy\\\nonumber
&&-\int_{\Omega}\nabla^k\partial_{y}\rho\nabla^k\partial_{y}u_{2}\partial_{y}\Bar{\Phi}dxdy-\int_{\Omega}\nabla^k\partial_{x}\rho\nabla^k\partial_{x}u_{2}\partial_{y}\Bar{\Phi}dxdy\\\nonumber
&&=-\int_{\Omega}\nabla^{k}\partial_{x}u_1\partial_{y}\bar{\Phi}\nabla^{k}\partial_{y}\rho\,dxdy+\int_{\Omega}\nabla^{k}\partial_{x}u_2\partial_{y}\bar{\Phi}\nabla^{k}\partial_{x}\rho\,dxdy\\\nonumber
&&-\int_{\Omega}\nabla^k\partial_{y}\rho\nabla^k\partial_{y}u_{2}\partial_{y}\Bar{\Phi}dxdy-\int_{\Omega}\nabla^k\partial_{x}\rho\nabla^k\partial_{x}u_{2}\partial_{y}\Bar{\Phi}dxdy=0,
\end{eqnarray}
and then
\begin{equation}\label{3.31}
I_{3}\lesssim E_k^2\|\bar{\Phi}\|_{H^{k+2}}.
\end{equation}

Similar to \eqref{3.26} and \eqref{3.30}, we also have
\begin{equation}\label{3.32}
I_{5}=I_6=0.
\end{equation}

For the term $I_4$, we have
\begin{eqnarray}\label{3.33}
&&I_{4}
\lesssim E_k^2(E_4 +\|\bar{\Phi}\|_{H^3}+\|\Psi\|_{H^3})+E_k E_4(\|\bar{\Phi}\|_{H^{k+2}} +\| \Psi\|_{H^{k+2}} )\\\nonumber
&&+<\nabla^{k}\partial_{t}u,\partial_{y}\Phi\nabla^{k}\partial_{t}(\partial_{xy}\rho,-\partial_{x}^{2}\rho)^T>-<\nabla^k\nabla^{\bot}\partial_{t}\rho,\nabla^k\nabla^{\bot}\partial_{t}u_{2}\partial_{y}\Bar{\Phi}>\\\nonumber
&&-<\nabla^k\nabla^{\bot}\partial_{t}\rho,\nabla^{k}\nabla^{\bot}u_2\partial_{t}\partial_{y}\Phi>+<\nabla^{k}\partial_{t}u,\partial_{t}\partial_{y}\Phi\nabla^{k}(\partial_{xy}\rho,-\partial_{x}^{2}\rho)^T>
.
\end{eqnarray}

Again,  by the incompressible condition and integrating by parts, similar to \eqref{3.30},  we have
\begin{eqnarray}\label{3.34}
&&<\nabla^{k}\partial_{t}u,\partial_{y}\Phi\nabla^{k}\partial_{t}(\partial_{xy}\rho,-\partial_{x}^{2}\rho)^T>-<\nabla^k\nabla^{\bot}\partial_{t}\rho,\nabla^k\nabla^{\bot}\partial_{t}u_{2}\partial_{y}\Bar{\Phi}>\\\nonumber
&&=\int_{\Omega}\partial_{y}\Bar{\Phi}\nabla^k\partial_{t}\partial_{xy}\rho\nabla^k\partial_{t}u_{1}\,dxdy-\int_{\Omega}\partial_{y}\Bar{\Phi}\nabla^k\partial_{t}\partial_{x}^{2}\rho\nabla^k\partial_{t}u_{2}\,dxdy\\\nonumber
&&-\int_{\Omega}\nabla^k\partial_{ty}u_{2}\partial_{y}\Bar{\Phi}\nabla^k\partial_{yt}\rho\,dxdy-\int_{\Omega}\nabla^k\partial_{tx}u_{2}\partial_{y}\Bar{\Phi}\nabla^k\partial_{xt}\rho\,dxdy=0.
\end{eqnarray}
Next recalling \eqref{3.4},  we have
\begin{eqnarray}\label{3.35}
&&<\nabla^{k}\partial_{t}u_1,\partial_{t}\partial_{y}\Phi\nabla^{k}\partial_{xy}\rho>-
<\nabla^k\nabla^{\bot}\partial_{t}\rho,\nabla^{k}\nabla^{\bot}u_2\partial_{t}\partial_{y}\Phi>\\\nonumber
&&-<\nabla^{k}\partial_{t}u_2,\partial_{t}\partial_{y}\Phi\nabla^{k}\partial_{x}^{2}\rho)>\\\nonumber
&&\lesssim\frac{\|\partial_{t}u\|_{H^{k}}\|\partial_{t}\Phi\|_{H^{2}}}{(1-\|\partial_{x}\rho\|_{L^{\infty}})^{\frac{1}{2}}}(\int_{\Omega}(1+\partial_{x}\rho)|\nabla^{k+1}\nabla^{\bot}\rho|^{2}dxdy)^{\frac{1}{2}}\\\nonumber
&&\\\nonumber
&&+\frac{\|\partial_{t}\rho\|_{H^{k+1}}\|\partial_{t}\Phi\|_{H^{2}}}{(1-\|\partial_{x}\rho\|_{L^{\infty}})^{\frac{1}{2}}}
(\int_{\Omega}(1+\partial_{x}\rho)|\partial^{k+1}u|^{2}dxdy)^{\frac{1}{2}}.
\end{eqnarray}

Then from \eqref{3.33}-\eqref{3.35}, we get
\begin{eqnarray}\label{3.36}
&&I_{4}\lesssim(E_k^2+\Gamma_{k+1}^2)(E_4 +\|\bar{\Phi}\|_{H^{k+2}}+\|\Psi\|_{H^{k+2}}).
\end{eqnarray}

Combing \eqref{3.24}, \eqref{3.28}, \eqref{3.31}, \eqref{3.32} with \eqref{3.36}, we get
\begin{equation}\label{3.37}
G_3\lesssim(E_k^2+\Gamma_{k+1}^2)(E_4 +\|\bar{\Phi}\|_{H^{k+2}}+\|\Psi\|_{H^{k+2}}).
\end{equation}

\end{proof}

\begin{proposition}
Under the assumptions of Proposition 3.2, we have  the following weighted  estimates:
\begin{eqnarray}\label{3.38}
&&\frac{d}{dt} \Gamma_{k+1}^2+\kappa\int_{\Omega}(1+\partial_{x}\rho)|\partial^{k+1}u|^2dxdy\\\nonumber
&&\lesssim(E_k^2+\Gamma_{k+1}^2)(E_4+E_4^2+\|\bar{\Phi}\|_{H^{k+3}}+\|\Psi\|_{H^{k+3}}).
\end{eqnarray}
\end{proposition}

\begin{proof}
Due to the equation \eqref{1.7}, we have
\begin{eqnarray}\label{3.39}
&& \\\nonumber
&&\frac{1}{2}\frac{d}{dt}\int_{\Omega}|\partial^{k+1}u|^{2}(1+\partial_{x}\rho)dxdy
\\\nonumber
&&=\int_{\Omega}(1+\partial_{x}\rho)\partial^{k+1}u\cdot\partial^{k+1}[-(u\cdot\nabla)u-(u_{2}\partial_{y}\Psi,0)^T-\kappa{u}
\\\nonumber
&&
-\Psi\partial_{x}u-\nabla{P}+(\nabla^{\bot}\rho\cdot\nabla)\nabla^{\bot}\rho-(\partial_{x}\rho\partial_{y}^{2}\Bar{\Phi},0)^T
\\\nonumber
&&
+(\partial_{y}\Bar{\Phi}\partial_{xy}\rho,-\partial_{y}\Bar{\Phi}\partial_{x}^{2}\rho)^T-(C_0\partial_{y}^{2}\rho,-C_0\partial_{xy}\rho)^T]dxdy
+\frac{1}{2}\int_{\Omega}|\partial^{k+1}u|^{2}\partial_{y}\partial_{t}\rho\,dxdy.
\end{eqnarray}
Similarly, we also have:
\begin{eqnarray}\label{3.40}
&& \\\nonumber
&&\frac{1}{2}\frac{d}{dt}\int_{\Omega}|\partial^{k+1}\nabla^{\bot}\rho|^{2}(1+\partial_{x}\rho)dxdy
\\\nonumber
&&
=\int_{\Omega}(1+\partial_{x}\rho)\partial^{k+1}\nabla^{\bot}\rho\cdot \partial^{k+1}\nabla^{\bot}(-u_{2}\partial_{y}\Bar{\Phi}-C_0u_{1}-(u\cdot\nabla)\rho
-\Psi\partial_{x}\rho)dxdy\\\nonumber
&&+\frac{1}{2}\int_{\Omega}|\partial^{k+1}\nabla^{\bot}\rho|^2\partial_{xt}\rho\,dxdy.
\end{eqnarray}
Then \eqref{3.39}-\eqref{3.40} imply that
\begin{eqnarray}\label{3.41}
&&\frac{1}{2}\frac{d}{dt}\Gamma_{k+1}^2 +\kappa\int_{\Omega}(1+\partial_{x}\rho)|\partial^{k+1}u|^2dxdy
\\\nonumber
&&
\lesssim \int_{\Omega}(1+\partial_{x}\rho)\partial^{k+1}u\cdot\partial^{k+1}[-(u\cdot\nabla)u-(u_{2}\partial_{y}\Psi,0)^T
\\\nonumber
&&
-\Psi\partial_{x}u-(\partial_{x}\rho\partial_{y}^{2}\Bar{\Phi},0)^T]dxdy
\\\nonumber
&&
+\int_{\Omega}(1+\partial_{x}\rho)\partial^{k+1}\nabla^{\bot}\rho\cdot \partial^{k+1}\nabla^{\bot}(
-\Psi\partial_{x}\rho)dxdy\\\nonumber
&&
-\int_{\Omega}(1+\partial_{x}\rho)\partial^{k+1}u\cdot\nabla\partial^{k+1}{P}dxdy+\frac{\|\partial_{tx}\rho\|_{L^{\infty}(\Omega)}}{1-\|\partial_{x}\rho\|_{L^{\infty}}}\Gamma_{k+1}^2+\sum_{i=7}^9 I_i,
\end{eqnarray}
with
\begin{align}
&I_7=\int_{\Omega}(1+\partial_{x}\rho)\partial^{k+1}u\cdot\partial^{k+1}(\partial_{y}\Bar{\Phi}\partial_{xy}\rho,-\partial_{y}\Bar{\Phi}\partial_{x}^{2}\rho)^Tdxdy\\\nonumber
&-\int_{\Omega}(1+\partial_{x}\rho)
\partial^{k+1}\nabla^{\bot}\rho\cdot\partial^{k+1}\nabla^{\bot}(u_{2}\partial_{y}\Bar{\Phi})]dxdy,\\
&I_8=\int_{\Omega}(1+\partial_{x}\rho)\partial^{k+1}u\cdot\partial^{k+1}
[(\nabla^{\bot}\rho\cdot\nabla)\nabla^{\bot}\rho]dxdy\\\nonumber
&-\int_{\Omega}(1+\partial_{x}\rho)
\partial^{k+1}\nabla^{\bot}\rho\cdot\partial^{k+1}\nabla^{\bot}[(u\cdot\nabla)\rho]\,dxdy,\\
&I_9=C_0\int_{\Omega}(1+\partial_{x}\rho)\partial^{k+1}u\cdot\partial^{k+1}
(-\partial_{y}^{2}\rho,\partial_{xy}\rho)dxdy\\\nonumber
&-C_0\int_{\Omega}(1+\partial_{x}\rho)\partial^{k+1}\nabla^{\bot}\rho\cdot\partial^{k+1}\nabla^{\bot}u_{1}dxdy.
\end{align}

Now we estimate the R.H.S. of \eqref{3.41} term by term.
\begin{eqnarray}\label{3.45}
&&|\int_{\Omega}(1+\partial_{x}\rho)\partial^{k+1}u\cdot\partial^{k+1}[(u\cdot\nabla)u]dxdy|\\\nonumber
&&\lesssim
\Gamma_{k+1}E_kE_4+ |\int_{\Omega}(1+\partial_{x}\rho)\partial^{k+1}u\cdot(\partial^{k+1}u\cdot\nabla)udxdy|\\\nonumber
&&+|\int_{\Omega}(1+\partial_{x}\rho)\partial^{k+1}u\cdot(u\cdot\nabla)\partial^{k+1}udxdy|\lesssim (\Gamma_{k+1}^2+E_{k}^2)E_4.
\end{eqnarray}

Similarly, we have
\begin{eqnarray}\label{3.46}
&&|\int_{\Omega}(1+\partial_{x}\rho)\partial^{k+1}u_{1}\cdot\partial^{k+1}(u_{2}\partial_{y}\Psi)dxdy|\\\nonumber
&&\lesssim |\int_{\Omega}(1+\partial_{x}\rho)\partial^{k+1}u_{1}\cdot\partial^{k+1}u_{2}\partial_{y}\Psi\,dxdy|\\\nonumber
&&+|\int_{\Omega}(1+\partial_{x}\rho)\partial^{k+1}u_{1}\cdot{u}_{2}\partial^{k+1}\partial_{y}\Psi\,dxdy|+(E_k^2+\Gamma_{k+1}^2) \|\Psi\|_{H^{3}}
\\\nonumber
&&\lesssim(E_k^2+\Gamma_{k+1}^2)(\|\Psi\|_{H^{k+2}}+E_4)(1+\|\partial_x\rho\|_{H^2})^{\frac12}.
\end{eqnarray}

Integrating by parts, we have:
\begin{eqnarray}
&&|\int_{\Omega}(1+\partial_{x}\rho)\partial^{k+1}\nabla^{\bot}\rho\cdot \Psi\partial^{k+1}\nabla^{\bot}\partial_{x}{\rho}dxdy|\\\nonumber
&&=|\int_{\Omega}\partial_{x}^{2}\rho\Psi(\partial^{k+1}\nabla^{\bot}\rho)^{2}dxdy|\lesssim\frac{\|\rho\|_{H^{4}}\|\Psi\|_{H^{2}}}
{1-\|\partial_{x}\rho\|_{L^{\infty}}}\Gamma_{k+1}^2,
\end{eqnarray}
therefore we have
\begin{eqnarray}
&&|\int_{\Omega}(1+\partial_{x}\rho)\partial^{k+1}\nabla^{\bot}\rho\cdot \partial^{k+1}\nabla^{\bot}(-\Psi\partial_{x}\rho)dxdy|\\\nonumber
&&\lesssim\Gamma_{k+1}(\|\Psi\|_{H^{k+2}}\|\rho\|_{H^4}+\|\Psi\|_{H^3}\|\rho\|_{H^{k+1}}+\Gamma_{k+1}\|\Psi\|_{H^2})\\\nonumber
&&+|\int_{\Omega}(1+\partial_{x}\rho)\partial^{k+1}\nabla^{\bot}\rho\cdot \Psi\partial^{k+1}\nabla^{\bot}\partial_{x}{\rho}dxdy|.
\end{eqnarray}

Similarly, we have
\begin{eqnarray}
&& \\\nonumber
&&\int_{\Omega}(1+\partial_{x}\rho)\partial^{k+1}\nabla^{\bot}\rho\cdot\partial^{k+1}\nabla^{\bot}(-\Psi\partial_{x}\rho)dxdy\\\nonumber
&&\lesssim\Gamma_{k+1}[\|\Psi\|_{H^{k+2}}\|\nabla\rho\|_{H^3}+\|\Psi\|_{H^3}\|\rho\|_{H^{k+1}}+\Gamma_{k+1}(\|\Psi\|_{H^2}+\|\rho\|_{H^4}\|\Psi\|_{H^2})],
\end{eqnarray}
and
\begin{eqnarray}
&&|\int_{\Omega}(1+\partial_{x}\rho)\partial^{k+1}u\cdot \partial^{k+1}(\Psi\partial_{x}u)-(1+\partial_{x}\rho)\partial^{k+1}u_1\partial^{k+1}(\partial_{x}\rho\partial_{y}^{2}\Bar{\Phi})dxdy\\\nonumber
&&
\lesssim (\|u\|_{H^k}\|\Psi\|_{H^3}+\|\rho\|_{H^{k+1}}\|\bar{\Phi}\|_{H^4}+\|\bar{\Phi}\|_{H^{k+3}}\|\rho\|_{H^3})\Gamma_{k+1}\\\nonumber
&&+\int_{\Omega}(1+\partial_{x}\rho)\partial^{k+1}u\cdot \big( \partial^{k+1}\Psi\partial_{x}u+\Psi\partial^{k+1}\partial_{x}u\big)dxdy\\\nonumber
&&-\int_{\Omega}(1+\partial_{x}\rho)\partial^{k+1}u_1\partial^{k+1}\partial_{x}\rho\partial_{y}^{2}\Bar{\Phi}dxdy\\\nonumber
&&\lesssim\Gamma_{k+1}(\|\Psi\|_{H^{k+1}}\|u\|_{H^3}+\|u\|_{H^{k}}\|\Psi\|_{H^3}+\Gamma_{k+1}\|\Psi\|_{H^2}\|\rho\|_{H^4})\\\nonumber
&&+\Gamma_{k+1}(\Gamma_{k+1}\|\Phi\|_{H^{3}}+\|\Phi\|_{H^{k+3}}\|\nabla\rho\|_{H^2}+\|\rho\|_{H^{k+1}}\|\bar{\Phi}\|_{H^{4}})\\\nonumber
&&
\lesssim(\Gamma_{k+1}^2+E_k^2)(E_4+\|\Psi\|_{H^{k+1}}+\|\bar{\Phi}\|_{H^{k+3}}).
\end{eqnarray}

By using the equation \eqref{1.7} and the incompressible condition, we get
\begin{equation}\label{3.51}
P=-\Delta^{-1}\nabla\cdot[(u\cdot\nabla)u+(u_2\partial_{y}\Psi,0)^T+\Psi\partial_xu-(\nabla^\bot\rho\cdot)\nabla^\bot\rho]
.\end{equation}
Then by using Lemma 2.1, we have
\begin{eqnarray}
&&\|\nabla P\|_{\dot{H}^k}\leq \|(u\cdot\nabla)u+u_2\partial_{y}\Psi+\Psi\partial_xu-(\nabla^\bot\rho\cdot\nabla)\nabla^\bot\rho\|_{\dot{H}^k}
\\\nonumber
&&
\lesssim(\|u\|_{H^k}+\|\rho\|_{H^{k+1}})E_4+\|\Psi\|_{H^{k+1}}E_k
\\\nonumber
&&
+\frac{\|\nabla\rho\|_{L^\infty}}{\sqrt{1-\|\partial_x\rho\|_{L^\infty}}}\|\sqrt{1+\partial_x\rho}\nabla^{k+1}\nabla^\bot \rho\|_{L^2}
\\\nonumber
&&
\lesssim  E_4E_k+E_4\Gamma_{k+1} +\|\Psi\|_{H^{k+1}}E_k.
\end{eqnarray}
Thus, we get
\begin{eqnarray}
&&|\int_{\Omega}(1+\partial_{x}\rho)\partial^{k+1}u\cdot\nabla^{k+1}\nabla{P}dxdy|=|\int_{\Omega}\nabla\partial_{x}\rho\partial^{k+1}u\cdot\nabla^k\nabla{P}dxdy|\\\nonumber
&&\lesssim
\|\nabla\rho\|_{H^3}\frac{1}{\sqrt{1-\|\partial_x\rho\|_{L^\infty}}}\|\sqrt{1+\partial_x\rho} \nabla^{k+1}u\|_{L^2}\|\nabla P\|_{H^k}
\\\nonumber
&&
\lesssim E_4\Gamma_{k+1}(E_k+\Gamma_{k+1})(E_4+\|\Phi\|_{H^{k+1}}+\|\Psi\|_{H^{k+1}}).
\end{eqnarray}

As to the term  $I_{7}$, there holds
\begin{eqnarray}
&& \\\nonumber
&&I_{7}\lesssim\Gamma_{k+1}(\|\Phi\|_{H^{k+2}}\|\rho\|_{H^{4}}+\|\Phi\|_{H^{k+3}}\|u\|_{H^{3}}+\|u\|_{H^{k}}\|\Phi\|_{H^{4}}+\|\Phi\|_{H^{4}}\|\rho\|_{H^{k+1}})\\\nonumber
&&+\Gamma_{k+1}^2\|\Phi\|_{H^{3}}+|\int_{\Omega}(1+\partial_{x}\rho)\partial^{k+1}u\cdot(\partial_{y}\Bar{\Phi}\partial^{k+1}\partial_{xy}\rho,-\partial_{y}\Bar{\Phi}\partial^{k+1}\partial_{x}^{2}\rho)^Tdxdy\\\nonumber
&&-\int_{\Omega}(1+\partial_{x}\rho)\partial^{k+1}\nabla^{\bot}\rho\cdot(\partial^{k+1}\nabla^{\bot}u_{2}\partial_{y}\Bar{\Phi})dxdy|.
\end{eqnarray}
Integrating by parts, we have
\begin{eqnarray}\label{3.55}
&& \\\nonumber
&&\int_{\Omega}(1+\partial_{x}\rho)\partial^{k+1}u\cdot(\partial_{y}\Bar{\Phi}\partial^{k+1}\partial_{xy}\rho,-\partial_{y}\Bar{\Phi}\partial^{k+1}\partial_{x}^{2}\rho)^Tdxdy\\\nonumber
&&-\int_{\Omega}(1+\partial_{x}\rho)\partial^{k+1}\nabla^{\bot}\rho\cdot(\partial^{k+1}\nabla^{\bot}u_{2}\partial_{y}\Bar{\Phi})dxdy\\\nonumber
&&=-\int_{\Omega}(1+\partial_{x}\rho)\partial^{k+1}\partial_{x}u_{1}\partial_{y}\Bar{\Phi}\partial^{k+1}\partial_{y}{\rho}dxdy-\int_{\Omega}\partial_{x}^{2}\rho\partial^{k+1}u_{1}\partial_{y}\bar{\Phi}\partial^{k+1}\partial_{y}{\rho}dxdy
\\\nonumber
&&+\int_{\Omega}(1+\partial_{x}\rho)\partial^{k+1}\partial_{x}u_{2}\partial_{y}\Bar{\Phi}\partial^{k+1}\partial_{x}{\rho}dxdy+\int_{\Omega}\partial_{x}^{2}\rho\partial^{k+1}u_{2}\partial_{y}\bar{\Phi}\partial^{k+1}\partial_{x}{\rho}dxdy\\\nonumber
&&-\int_{\Omega}(1+\partial_{x}\rho)\partial^{k+1}\partial_{y}\rho\partial^{k+1}\partial_{y}u_{2}\partial_{y}{\Bar{\Phi}}dxdy-\int_{\Omega}(1+\partial_{x}\rho)\partial^{k+1}\partial_{x}\rho\partial^{k+1}\partial_{x}u_{2}\partial_{y}{\Bar{\Phi}}dxdy\\\nonumber
&&=\int_{\Omega}\partial_{x}^{2}\rho\partial^{k+1}u_{2}\partial^{k+1}\partial_{x}{\rho}dxdy-\int_{\Omega}\partial_{x}^{2}\rho\partial^{k+1}u_{1}\partial^{k+1}\partial_{y}{\rho}dxdy
\lesssim\frac{\|\rho\|_{H^{4}}}{1-\|\partial_{x}\rho\|_{L^{\infty}}}\Gamma_{k+1}^2,
\end{eqnarray}
which implies that
\begin{eqnarray}
&& \\\nonumber
&&I_{7}\lesssim\Gamma_{k+1}(\|\Phi\|_{H^{k+2}}\|\rho\|_{H^{4}}+\|\Phi\|_{H^{k+3}}\|u\|_{H^{3}}+\|u\|_{H^{k}}\|\Phi\|_{H^{4}}+\|\Phi\|_{H^{4}}\|\rho\|_{H^{k+1}})\\\nonumber
&&+\Gamma_{k+1}^2(\|\Phi\|_{H^{3}}+\frac{\|\rho\|_{H^{4}}}{1-\|\partial_{x}\rho\|_{L^{\infty}}}).
\end{eqnarray}

Similar to \eqref{3.55}, we get the bound for $I_9$
\begin{eqnarray}
&& \\\nonumber
&&I_{9}=-C_0\int_{\Omega}(1+\partial_{x}\rho)\partial^{k+1}u_{1}\partial^{k+1}\partial_{y}^{2}{\rho}dxdy+C_0\int_{\Omega}(1+\partial_{x}\rho)\partial^{k+1}u_{2}\partial^{k+1}\partial_{xy}{\rho}dxdy\\\nonumber
&&-C_0\int_{\Omega}(1+\partial_{x}\rho)\partial^{k+1}\nabla^{\bot}\rho\cdot\partial^{k+1}\nabla^{\bot}u_{1}dxdy\\\nonumber
&&=C_0\int_{\Omega}(1+\partial_{x}\rho)\partial^{k+1}\partial_{y}u_{1}\partial^{k+1}\partial_{y}{\rho}dxdy+C_0\int_{\Omega}\partial_{xy}\rho\partial^{k+1}u_{1}\partial^{k+1}\partial_{y}{\rho}dxdy\\\nonumber
&&-C_0\int_{\Omega}(1+\partial_{x}\rho)\partial^{k+1}\partial_{y}u_{2}\partial^{k+1}\partial_{x}{\rho}dxdy-C_0\int_{\Omega}\partial_{xy}\rho\partial^{k+1}u_{2}\partial^{k+1}\partial_{x}{\rho}dxdy\\\nonumber
&&-C_0\int_{\Omega}(1+\partial_{x}\rho)\partial^{k+1}\partial_{y}\rho\partial^{k+1}\partial_{y}{u_{1}}dxdy-C_0\int_{\Omega}(1+\partial_{x}\rho)\partial^{k+1}\partial_{x}\rho\partial^{k+1}\partial_{x}{u_{1}}dxdy\\\nonumber
&&=C_0\int_{\Omega}\partial_{xy}\rho\partial^{k+1}u_{1}\partial^{k+1}\partial_{y}{\rho}dxdy-C_0\int_{\Omega}\partial_{xy}\rho\partial^{k+1}u_{2}\partial^{k+1}\partial_{x}{\rho}dxdy\\\nonumber
&&\lesssim\frac{\|\rho\|_{H^{4}}}{1-\|\partial_{x}\rho\|_{L^{\infty}}}\Gamma_{k+1}^2.
\end{eqnarray}

Now, let's turn to the term  $I_{8}$:
\begin{eqnarray}\label{3.58}
&&I_{8}\lesssim E_4\Gamma_{k+1}^2+\Gamma_{k+1}(\|\rho\|_{H^{k+1}}\|\nabla\rho\|_{H^4}+\|u\|_{H^{k}}\|\nabla\rho\|_{H^4}+\|\rho\|_{H^{k+1}}\|u\|_{H^{4}})\\\nonumber
&&+|\int_{\Omega}(1+\partial_{x}\rho)\partial^{k+1}u\cdot[(\nabla^{\bot}\rho\cdot\nabla)\partial^{k+1}\nabla^{\bot}\rho]dxdy\\\nonumber
&&-\int_{\Omega}(1+\partial_{x}\rho)\partial^{k+1}\nabla^{\bot}\rho\cdot(\partial^{k+1}\nabla^{\bot}u\cdot\nabla){\rho}dxdy|.
\end{eqnarray}
Integrating by part, we get
\begin{eqnarray}\label{3.59}
&&\int_{\Omega}(1+\partial_{x}\rho)\partial^{k+1}u\cdot[(\nabla^{\bot}\rho\cdot\nabla)\partial^{k+1}\nabla^{\bot}\rho]dxdy\\\nonumber
&&-\int_{\Omega}(1+\partial_{x}\rho)\partial^{k+1}\nabla^{\bot}\rho\cdot (\partial^{k+1}\nabla^{\bot}u\cdot\nabla){\rho}dxdy\\\nonumber
&&=\int_{\Omega}(1+\partial_{x}\rho)\partial^{k+1}u_{1}(\partial_{y}\rho\partial^{k+1}\partial_{xy}\rho-\partial_{x}\rho\partial^{k+1}\partial_{y}^{2}\rho)dxdy\\\nonumber
&&+\int_{\Omega}(1+\partial_{x}\rho)\partial^{k+1}u_{2}(\partial_{x}\rho\partial^{k+1}\partial_{xy}\rho-\partial_{y}\rho\partial^{k+1}\partial_{x}^{2}\rho)\\\nonumber
&&-\int_{\Omega}(1+\partial_{x}\rho)\partial^{k+1}\partial_{y}\rho(\partial^{k+1}\partial_{y}u_{1}\partial_{x}\rho+\partial^{k+1}\partial_{y}u_{2}\partial_{y}\rho)\\\nonumber
&&-\int_{\Omega}(1+\partial_{x}\rho)\partial^{k+1}\partial_{x}\rho(\partial^{k+1}\partial_{x}u_{1}\partial_{x}\rho+\partial^{k+1}\partial_{x}u_{2}\partial_{y}\rho)\lesssim E_4^2\Gamma_{k+1}^2
\end{eqnarray}

From \eqref{3.58}-\eqref{3.59} we get the bound  of $I_{8}\lesssim (E_4+E_4^2)(E_k^2+\Gamma_{k+1}^2)$.
Taking \eqref{3.39}-\eqref{3.59} as a whole,   we get
\begin{eqnarray}
&&\frac{d}{dt} \Gamma_{k+1}^2+\kappa\int_{\Omega}(1+\partial_{x}\rho)|\partial^{k+1}u|^2dxdy\\\nonumber
&&\lesssim(E_k^2+\Gamma_{k+1}^2)(E_4+E_4^2+\|\bar{\Phi}\|_{H^{k+3}}+\|\Psi\|_{H^{k+3}}).
\end{eqnarray}

\end{proof}

\section{Decay estimates}
In this section, our goal is to get the  decay rate in time of the perturbation $(u,\rho)$.
\subsection{Linearized problem}
At the beginning, our system \eqref{1.6}-\eqref{1.7} is rewritten as
\begin{equation}\label{4.1}
\begin{cases}
\partial_{t}u_{1}+\kappa u_{1}+C_0\partial_{y}^{2}\rho=F_{1}-\partial_{x}P,\\
\partial_{t}u_{2}+\kappa u_{2}-C_0\partial_{xy}\rho=F_{2}-\partial_{y}P,\\
\partial_{t}\rho+C_0u_{1}=F_{3}.
\end{cases}
\end{equation}
Here, $F_{1}$, $F_{2}$, and $F_{3}$ are
\begin{equation}\label{4.2}
\begin{cases}
F_{1}=(\nabla^{\bot}\rho\cdot\nabla)\partial_{y}\rho-\partial_{x}\rho\partial_{y}^{2}\bar{\Phi}+\partial_{y}\bar{\Phi}\partial_{xy}\rho-\Psi\partial_{x}u_{1}-u_{2}\partial_{y}\Psi-(u\cdot\nabla)u_{1},\\
F_{2}=-(\nabla^{\bot}\rho\cdot\nabla)\partial_{x}\rho-\partial_{y}\bar{\Phi}\partial_{x}^{2}\rho-\Psi\partial_{x}u_{2}-(u\cdot\nabla)u_{2},\\
F_{3}=-u_{2}\partial_{y}\bar{\Phi}-(u\cdot\nabla)\rho-\bar{\Phi}\partial_{x}\rho.
\end{cases}
\end{equation}

Recalling the incompressible condition, $P$ is calculated as  \eqref{3.51}, the linearized equation of \eqref{4.1} is:
\begin{equation}\label{4.3}
\begin{cases}
\partial_{t}u_{1}+\kappa u_{1}=-C_0\partial_{y}^{2}\rho,\\
\partial_{t}u_{2}+\kappa u_{2}=C_0\partial_{xy}\rho,\\
\partial_{t}\rho+C_0u_{1}=0.
\end{cases}
\end{equation}

Decoupling the system \eqref{4.3}, we get
\begin{equation}\label{4.4}
\begin{cases}
\partial_{t}^{2}u_{1}+\kappa\partial_{t}u_{1}-C_0^{2}\partial_{y}^{2}u_1=0,\\
\partial_{t}^{2}u_{2}+\kappa\partial_{t}u_{2}-C_0^{2}\partial_{y}^2u_2=0,\\
\partial_{t}^{2}\rho+\kappa\partial_{t}\rho-C_0^{2}\partial_{y}^{2}\rho=0,
\end{cases}
\end{equation}
with the initial data $u|_{t=0}=(u_{10},u_{20})$ and  $\rho|_{t=0}=\rho_0$. Besides,   boundary conditions are
$u_2|_{y=0,1}=0$ and $\partial_y u_1|_{y=0,1}=0.$

Denoting $\mathscr{F}$ as  the Fourier transform, and noting   boundary conditions, we write
\begin{equation}\label{4.5}
(\mathcal{F}_{b_q}(u_{10}),\mathcal{F}_{a_q}(u_{20}),\mathcal{F}_{b_q}(\rho_0))(\xi,q)
=\mathscr{F}(u_{10},u_{20},\rho_0), \quad (\xi,q)\in \mathbb{R}\times (\mathbb{N}\cup \{0\}).
\end{equation}

From the system \eqref{4.3}, we have
\begin{eqnarray}\label{4.6}
&&\mathcal{F}_{b_q}(\partial_{t}u_{1}|_{t=0})(\xi,q)=-\kappa\mathcal{F}_{b_q}(u_{10})(\xi,q)+\frac{C_0q^2\pi^2}{4}\mathcal{F}_{b_q}(\rho_0)(\xi,q),\\\nonumber
&&\mathcal{F}_{a_q}(\partial_{t}u_{2}|_{t=0})(\xi,q)=-\kappa\mathcal{F}_{a_q}(u_{20})(\xi,q)-C_0\frac{q\xi\pi}{2}\mathcal{F}_{b_q}(\rho_0)(\xi,q),\\\nonumber
&&\mathcal{F}_{b_q}(\partial_{t}\rho|_{t=0})(\xi,q)=-C_0\mathcal{F}_{b_q}(u_{10})(\xi,q).
\end{eqnarray}

By applying the Fourier transform  $\mathscr{F}$ to \eqref{4.4} and noting   initial conditions \eqref{4.5}-\eqref{4.6}, it is sufficient to study the following ODE
\begin{equation}\label{4.7}
\frac{d^2}{dt^2}\widehat{\Upsilon}(\xi,q,t)+\kappa\frac{d}{dt}\widehat{\Upsilon}(\xi,q,t)+\frac{C_0^{2}\pi^2}4q^{2}\widehat{\Upsilon}(\xi,q,t)=0, \quad q\in \mathbb{N}\cup \{0\}.
\end{equation}

$\widehat{\Upsilon}$ stands for the Fourier transform of $\Upsilon$. Write
\begin{equation}\label{4.8}
\delta(\kappa,C_0,q)=\kappa^2-C_0^{2}q^2\pi^2,
\end{equation}
the solution of \eqref{4.7} is given by
\begin{equation}\label{4.9}
\widehat{\Upsilon}(\xi,q,t)=
\begin{cases}
\frac{\widehat{\Upsilon}^{\prime}(\xi,q,0)-\phi_{-}\widehat{\Upsilon}(\xi,q,0)}{\sqrt{\delta}}e^{\phi_{+}t}
+\frac{\phi_{+}\widehat{\Upsilon}(\xi,q,0)
-\widehat{\Upsilon}^{\prime}(\xi,q,t)}{\sqrt{\delta}}e^{\phi_{-}t},\quad\delta\neq0,\\
\widehat{\Upsilon}(\xi,q,0)e^{-\frac{\kappa}{2}t}+[\frac{\kappa}{2}\widehat{\Upsilon}(\xi,q,0)+\widehat{\Upsilon}^{\prime}(\xi,q,0)]
te^{-\frac{\kappa}{2}t},\quad\delta=0,
\end{cases}
\end{equation}
where $\widehat{\Upsilon}^{\prime}$ is $\frac{d}{dt}\widehat{\Upsilon}$,  and
$
\phi_{\pm}(q)=\frac{-\kappa\pm\sqrt{\delta}}{2}.
$


Therefore,  solutions to \eqref{4.3} with initial data \eqref{4.5}-\eqref{4.6}, we have,
\begin{eqnarray}\label{4.10}
\mathcal{F}_{b_q}(u_{1}(t))(\xi,q)&=&\frac{C_0(q\frac{\pi}{2})^2\mathcal{F}_{b_q}(\rho_0)+\phi_{+}\mathcal{F}_{b_q}(u_{10})}{\sqrt{\delta}}e^{\phi_{+}t}
\\\nonumber
&&
-\frac{C_0(q\frac{\pi}{2})^2\mathcal{F}_{b_q}(\rho_0)+\phi_{-}\mathcal{F}_{b_q}(u_{10})}{\sqrt{\delta}}e^{\phi_{-}t},
\end{eqnarray}
\begin{eqnarray}\label{4.11}
\mathcal{F}_{b_q}(\rho(t))(\xi,q))&=&\frac{-C_0\mathcal{F}_{b_q}(u_{10})-\phi_{-}\mathcal{F}_{b_q}(\rho_0)}{\sqrt{\delta}}e^{\phi_{+}t}
\\\nonumber
&&+\frac{\phi_{+}\mathcal{F}_{b_q}(\rho_0)
+C_0\mathcal{F}_{b_q}(u_{10})}{\sqrt{\delta}}e^{\phi_{-}t},
\end{eqnarray}
\begin{eqnarray}\label{4.12}
\mathcal{F}_{a_q}(u_{2}(t))(\xi,q)&=&\frac{\phi_{+}\mathcal{F}_{a_q}(u_{20})-C_0\xi{q}\frac{\pi}{2}\mathcal{F}_{b_q}(\rho_0)}{\sqrt{\delta}}e^{\phi_{+}t}
\\\nonumber
&&+\frac{C_0\xi{q}\frac{\pi}{2}\mathcal{F}_{b_q}(\rho_0)-\phi_{-}\mathcal{F}_{a_q}(u_{20})}{\sqrt{\delta}}e^{\phi_{-}t},
\end{eqnarray}
where we use the fact $-\kappa-\phi_{-}(\xi,q)=\phi_{+}$.

\subsection{Linear Decay.}

From \eqref{4.8} and \eqref{4.10}-\eqref{4.12},  we know that the linearized  solution of \eqref{4.3} does not decay when  $\delta=\kappa^2$, which is equivalent to
the case $q=0.$  Due to the initial  conditions
\begin{equation}\label{4.13}
\int_0^1u_{10}(\xi,y)dy=\int_0^1\rho_0 (\xi,y)dy=0,
\end{equation}
from \eqref{4.10}, we have
\begin{equation}\label{4.14}
\mathcal{F}_{b_q}(u_{1}(t))(\xi,0)=\mathcal{F}_{b_q}(u_{1}(t))(\xi,0)e^{-\kappa{t}}=0.
\end{equation}

Similarly, from \eqref{4.11} we have
\begin{eqnarray}\label{4.15}
&&\mathcal{F}_{b_q}(\rho(t))(\xi,0)=\frac{C_0}{\kappa}\mathcal{F}_{b_q}(u_{10})e^{-\kappa{t}}+\mathcal{F}_{b_q}(\rho_0(\xi,0)-\frac{C_0}{\kappa}\mathcal{F}_{b_q}(u_{10})\\\nonumber
&&=\frac{C_0}{\kappa}\mathcal{F}_{b_q}(u_{10})(\xi,0)(e^{-\kappa{t}}-1)+\mathcal{F}_{b_q}(\rho_0(\xi,0)=0,
\end{eqnarray}
as well as
\begin{equation}\label{4.16}
\mathcal{F}_{a_q}(u_{2}(t))(\xi,0)=\mathcal{F}_{a_q}(u_{20})(\xi,0)e^{-\kappa t}.
\end{equation}

From \eqref{4.14}-\eqref{4.16}, we conclude the exponential decay when   initial  conditions \eqref{4.13} satisfies for $q=0$.

When $q>0$, there exists a uniform constant $c_\kappa>0$
\begin{equation}\label{4.17}
-\frac{\kappa}{2}<\sup_{q\in\mathbb{N}}Re \phi_{+}=\sup_{q\in\mathbb{N}}Re\frac{-\kappa+\sqrt{\kappa^{2}-C_0^{2}q^2\pi^2}}{2}=\frac{-\kappa+\sqrt{\kappa^{2}-C_0^{2}\pi^2}}{2}\triangleq-c_\kappa<0.
\end{equation}

Taking  \eqref{4.13}-\eqref{4.17} as whole, we have the following decay estimates for the linearized system \eqref{4.3}.
\begin{lemma}\label{lemma6}
If $(\rho,u)$ is the solution of \eqref{4.3} with the condition  \eqref{1.7} and \eqref{4.13}, then we have
\begin{equation}
\begin{cases}
\|\rho\|_{H^{n}(\Omega)}(t)\leq e^{-c_\kappa t}[\|\rho_0\|_{H^{n}(\Omega)} +\|u_{10}\|_{H^{n}(\Omega)} ],\\
\|u_{1}\|_{H^{n}(\Omega)}(t)\leq e^{-c_\kappa t}[\|\rho_0\|_{H^{n}(\Omega)}+\|u_{10}\|_{H^{n}(\Omega)} ],\\
\|u_{2}\|_{H^{n}(\Omega)}(t)\leq e^{-c_\kappa t}[\|\rho_0\|_{H^{n}(\Omega)}+\|u_{20}\|_{H^{n}(\Omega)}].
\end{cases}
\end{equation}
\end{lemma}

\subsection{Non-Linear Decay.}
Next, by using the  Duhamel's principle, we denote
\begin{equation}
(w_{1}, w_{2}, w_{3})(t,\xi,q):=(\mathcal{F}_{b_q}u_{1},\mathcal{F}_{a_q}u_{2},\mathcal{F}_{b_q}\rho),
\end{equation}
and
\begin{equation}
(Q_{1},Q_{2},Q_{3})(t,\xi,q):=(\mathcal{F}_{b_q}F_{1},\mathcal{F}_{a_q}F_{2},\mathcal{F}_{b_q}F_{3}),
\end{equation}
with $F_{1},F_{2},F_{3}$ defined by \eqref{4.2}.

Our system \eqref{4.1} can be rewritten as:
\begin{equation}\label{4.21}
w^{\prime}(t)=A\cdot\,w(t)+Q(t),
 \end{equation}
with
\begin{equation}
 w(t)= \begin{pmatrix}
   w_{1} \\
  w_{2}\\
   w_{3} \\
  \end{pmatrix}
  , A=\left (
 \begin{matrix}
   -\kappa & 0 & C_0\frac{q^2\pi^2}{4} \\
   0& -\kappa &-C_0\xi{q}\frac{\pi}{2}\\
   -C_0 & 0 & 0
  \end{matrix}
  \right)
 ,Q(t)=\begin{pmatrix}
   Q_{1}\\
  Q_{2}\\
   Q_{3}
  \end{pmatrix}.
 \end{equation}

The solution of system \eqref{4.4} should  be
\begin{equation}
w(t)=e^{At}w(0)+\int_{0}^{t}e^{A(t-s)}Q(s)ds.
\end{equation}

Moreover, we get eigenvalues of the matrix $A$ by letting
\begin{equation}
|\lambda I-A|=
\begin{vmatrix}
\lambda+k&0& -C_0\frac{q^2\pi^2}{4}\\
0&\lambda+k&C_0\xi{q}\frac{\pi}{2}\\
C_0&0&\lambda\\
\end{vmatrix}=0,
\end{equation}
from which we get the eigenvalue of $A$ as follows
\begin{equation}
\lambda_{\pm}=\frac{-\kappa\pm\sqrt{\kappa^{2}-C_0^{2}q^2\pi^2}}{2}  ,\qquad \qquad \lambda=-\kappa.
\end{equation}

And then we  take the eigenvectors as
\begin{equation}
\begin{cases}
v_{+}=(\frac{C_0(q\frac{\pi}{2})^{2}}{-\lambda_{-}},\frac{C_0\xi{q}\frac{\pi}{2}}{\lambda_{-}},1)^T,\\
v_{-}=(\frac{C_0(q\frac{\pi}{2})^{2}}{-\lambda_{+}},\frac{C_0\xi{q}\frac{\pi}{2}}{\lambda_{+}},1)^T,\\
v=(0,1,0)^T.
\end{cases}
\end{equation}

As proved in the linear system part, different results will appear when $\delta$  vanish or not.
We shall discuss the decay rate in the following two cases:

\textbf{Case 1}(  $\delta\neq 0$):
Denote
\begin{equation}
S=(v_{+}, v_{-}, v)=\left (
 \begin{matrix}
\frac{C_0(q\frac{\pi}{2})^{2}}{-\lambda_{-}}& \frac{C_0(q\frac{\pi}{2})^{2}}{-\lambda_{+}} & 0 \\
\frac{C_0\xi{q}\frac{\pi}{2}}{\lambda_{-}}& \frac{C_0\xi{q}\frac{\pi}{2}}{\lambda_{+}} & 1\\
1&1& 0
  \end{matrix}
  \right),
\end{equation}
  and
\begin{equation}
D=\left (
 \begin{matrix}
   \lambda_{+}(\xi,q)& 0 &0 \\
   0&\lambda_{-}(\xi,q) &0\\
   0 & 0 & \lambda
  \end{matrix}
  \right),
\end{equation}
 there holds $A=S DS^{-1}$, with $S^{-1}$  given by:
\begin{equation}
S^{-1}=\left (
 \begin{matrix}
\frac{\lambda_{-}\lambda_{+}}{C_0(q\frac{\pi}{2})^{2}(-\sqrt{\delta})}&0&\frac{-\lambda_{-}}{\sqrt{\delta}} \\
\frac{\lambda_{-}\lambda_{+}}{C_0(q\frac{\pi}{2})^{2}(\sqrt{\delta})}&0& \frac{\lambda_{+}}{\sqrt{\delta}}\\
\frac{\xi{q}\frac{\pi}{2}}{(q\frac{\pi}{2})^{2}}&1& 0
  \end{matrix}
  \right).
\end{equation}

Then we have
\begin{equation}
\begin{aligned}
&e^{At}=S\left (
 \begin{matrix}
e^{\lambda_{+}t}&0&0\\
0&e^{\lambda_{-}t}&0\\
0&0&e^{\lambda t}
  \end{matrix}
  \right)S^{-1}\\
&=\left (
 \begin{matrix}
e^{\lambda_{+}t}\frac{\lambda_{+}}{\sqrt{\delta}}-e^{\lambda_{-}t}\frac{\lambda_{-}}{\sqrt{\delta}}&0&\frac{C_0(q\frac{\pi}{2})^{2}}{\sqrt{\delta}}(e^{\lambda_{+}t}-e^{\lambda_{-}t})\\
e^{\lambda_{+}t}\frac{-\xi\lambda_{+}}{q\frac{\pi}{2}\sqrt{\delta}}+e^{\lambda_{-}t}\frac{\xi\lambda_{-}}{q\frac{\pi}{2}\sqrt{\delta}}+e^{\lambda{t}}\frac{\xi}{q\frac{\pi}{2}}&e^{\lambda t}&-\frac{C_0\xi{q}\frac{\pi}{2}}{\sqrt{\delta}}(e^{\lambda_{+}t}-e^{\lambda_{-}t})\\
\frac{\lambda_{-}\lambda_{+}}{C_0(q\frac{\pi}{2})^{2}(-\sqrt{\delta})}(e^{\lambda_{+}t}-e^{\lambda_{-}t})&0&\frac{-\lambda_{-}}{\sqrt{\delta}}e^{\lambda_{+}}+\frac{\lambda_{+}}{\sqrt{\delta}}e^{\lambda_{-}}
  \end{matrix}
  \right),
\end{aligned}
\end{equation}
which implies the following decay estimates
\begin{equation}\label{4.31}
\begin{aligned}
\mathcal{F}_{b_q}(u_{1}(t))&\approx\,e^{\lambda_{+}t}(\frac{\lambda_{+}}{\sqrt{\delta}}\mathcal{F}_{b_q}(u_{10})+\frac{C_0(q\frac{\pi}{2})^{2}}{\sqrt{\delta}}\mathcal{F}_{b_q}(\rho_0))\\
&+\frac{\lambda_{+}}{\sqrt{\delta}}\int_{0}^{t}e^{\lambda_{+}(t-s)}\mathcal{F}_{b_q}[F_{1}(s)]ds+\frac{C_0(q\frac{\pi}{2})^{2}}{\sqrt{\delta}}\int_{0}^{t}
e^{\lambda_{+}(t-s)}\mathcal{F}_{a_q}[F_{3}(s)]ds,\\
\mathcal{F}_{a_q}(u_{2}(t))&\approx\,e^{\lambda_{+}t}(\frac{-\xi}{q\frac{\pi}{2}}\frac{\lambda_{+}}{\sqrt{\delta}}\mathcal{F}_{b_q}(u_{10})-\frac{C_0\xi{q}\frac{\pi}{2}}{\sqrt{\delta}}\mathcal{F}_{b_q}(\rho_0))\\
&+\frac{-\xi}{q\frac{\pi}{2}}\frac{\lambda_{+}}{\sqrt{\delta}}\int_{0}^{t}e^{\lambda_{+}(t-s)}
\mathcal{F}_{b_q}[F_{1}(s)]ds-\frac{C_0\xi{q}\frac{\pi}{2}}{\sqrt{\delta}}\int_{0}^{t}e^{\lambda_{+}(t-s)}\mathcal{F}_{b_q}[F_{3}(s)]ds,\\
\mathcal{F}_{a_q}(\rho(t))&\approx\,e^{\lambda_{+}t}(\frac{\lambda_{-}\lambda_{+}}{C_0(q\frac{\pi}{2})^{2}(-\sqrt{\delta})}\mathcal{F}_{b_q}(u_{10})-\frac{\lambda_{-}}{\sqrt{\delta}}\mathcal{F}_{a_q}(u_{20}))\\
&-\frac{\lambda_{-}\lambda_{+}}{C_0(q\frac{\pi}{2})^{2}(\sqrt{\delta})}\int_{0}^{t}e^{\lambda_{+}(t-s)}\mathcal{F}_{b_q}[F_{1}(s)]ds
-\frac{\lambda_{-}}{\sqrt{\delta}}\int_{0}^{t}e^{\lambda_{+}(t-s)}\mathcal{F}_{a_q}[F_{3}(s)]ds.
\end{aligned}
\end{equation}

\textbf{Case 2} ($\delta=0$). In the case, $\lambda_{+}=\lambda_{-}=-\frac{\kappa}{2}$, taking
\begin{equation}\label{4.32}
S_1=\left(\begin{matrix}\frac{2C_0(q\frac{\pi}{2})^{2}}{\kappa}&\frac{(2\kappa-4)C_0(q\frac{\pi}{2})^{2}}
{\kappa^{2}}&0\\-\frac{2C_0\xi{q}\frac{\pi}{2}}{\kappa}&\frac{(4-2\kappa)C_0\xi{q}\frac{\pi}{2}}{\kappa^{2}}&1\\1&1&0\end{matrix}\right),
\end{equation}
and
\begin{equation}\label{4.33}
S_1^{-1}=\left(\begin{matrix} \frac{\kappa^{2}}{4C_0(q\frac{\pi}{2})^{2}} & 0& 1-\frac{\kappa}{2}\\ -\frac{\kappa^{2}}{4C_0(q\frac{\pi}{2})^2}& 0 & \frac{\kappa}{2} \\ \frac{\xi{q}\frac{\pi}{2}}{(q\frac{\pi}{2})^{2}} & 1 & 0 \end{matrix} \right),
\end{equation}
then we get $A=S_1JS_1^{-1}$  with the Jordan matrix J as
\begin{equation}\label{4.34}
J=\left(\begin{matrix} -\frac{\kappa}{2} & 1 & 0\\ 0 & -\frac{\kappa}{2} & 0 \\ 0 & 0 & -\kappa \end{matrix} \right).
\end{equation}

Now we write
\begin{equation}\label{4.35}
J=\left(\begin{matrix} -\frac{\kappa}{2} & 0 & 0\\ 0 & -\frac{\kappa}{2} & 0 \\ 0 & 0 & -\kappa \end{matrix} \right)
+\begin{pmatrix} 0 & 1 & 0\\ 0 &0& 0 \\ 0 & 0 &0 \end{pmatrix}=D+N,
\end{equation}
hence, we get
\begin{equation}\label{4.36}
e^{At}=e^{-\frac{\kappa}{2}t}\begin{pmatrix}1-\frac{\kappa}{2}t & 0 &C_0(q\frac{\pi}{2})^{2}t\\
\frac{-\kappa\xi}{2q\frac{\pi}{2}}t-\frac{\xi}{q\frac{\pi}{2}}(1-e^{-\frac{\kappa}{2}t})&e^{-\frac{\kappa}{2}t}&-C_0\xi{q}\frac{\pi}{2}t\\ -\frac{\kappa^{2}}{4C_0(q\frac{\pi}{2})^{2}}t & 0 &\frac{\kappa }{2}t+1 \end{pmatrix}.
\end{equation}

From \eqref{4.31} and \eqref{4.36}, we have:
\begin{equation}\label{4.37}
\begin{aligned}
\mathcal{F}_{b_q}(u_{1}(t))&\lesssim e^{-c_kt}[\frac{\lambda_{+}}{\sqrt{\delta}}\mathcal{F}_{b_q}(u_{10})+\frac{C_0(q\frac{\pi}{2})^{2}}{\sqrt{\delta}}\mathcal{F}_{b_q}(\rho_0)]\\
&+\int_{0}^{t}e^{-c_k(t-s)}\mathcal{F}_{b_q}[F_{1}(s)]ds+\int_{0}^{t}e^{-c_k(t-s)}\mathcal{F}_{a_q}[F_{3}(s)]ds,\\
\mathcal{F}_{a_q}(u_{2}(t))&\lesssim e^{-c_kt}[\frac{-\xi}{q\frac{\pi}{2}}\frac{\lambda_{+}}{\sqrt{\delta}}\mathcal{F}_{b_q}(u_{10})-\frac{C_0\xi{q}\frac{\pi}{2}}{\sqrt{\delta}}\mathcal{F}_{b_q}(\rho_0)]\\
&+\int_{0}^{t}e^{-c_k(t-s)}\mathcal{F}_{b_q}[F_{1}(s)]ds-\int_{0}^{t}e^{-c_k(t-s)}\mathcal{F}_{b_q}[F_{3}(s)]ds,\\
\mathcal{F}_{a_q}(\rho(t))&\lesssim e^{-c_kt}[\frac{\lambda_{-}\lambda_{+}}{C_0(q\frac{\pi}{2})^{2}(-\sqrt{\delta})}\mathcal{F}_{b_q}(u_{10})-\frac{\lambda_{-}}{\sqrt{\delta}}\mathcal{F}_{a_q}(u_{20})]\\
&+\int_{0}^{t}e^{-c_k(t-s)}\mathcal{F}_{b_q}[F_{1}(s)]ds-\int_{0}^{t}e^{-c_k(t-s)}\mathcal{F}_{a_q}[F_{3}(s)]ds.
\end{aligned}
\end{equation}

Furthermore, similar to the process of \eqref{4.31}-\eqref{4.36}, we also have
\begin{equation}\label{4.38}
\begin{aligned}
\mathcal{F}_{b_q}(\partial_tu_{1}(t))&\lesssim e^{-c_kt}[\frac{\lambda_{+}}{\sqrt{\delta}}\mathcal{F}_{b_q}(u_{10})+\frac{C_0(q\frac{\pi}{2})^{2}}{\sqrt{\delta}}\mathcal{F}_{b_q}(\rho_0)]+\mathcal{F}_{b_q}[F_{1}(t)]\\
&+\int_{0}^{t}e^{-c_k(t-s)}\mathcal{F}_{b_q}[F_{1}(s)]ds+\int_{0}^{t}e^{-c_k(t-s)}\mathcal{F}_{a_q}[F_{3}(s)]ds+\mathcal{F}_{a_q}[F_{3}(t)],\\
\mathcal{F}_{a_q}(\partial_tu_{2}(t))&\lesssim e^{-c_kt}[\frac{-\xi}{(q\frac{\pi}{2})}\frac{\lambda_{+}}{\sqrt{\delta}}\mathcal{F}_{b_q}(u_{10})-\frac{C_0\xi{q}
\frac{\pi}{2}}{\sqrt{\delta}}\mathcal{F}_{b_q}(\rho_0)]+\mathcal{F}_{b_q}[F_{1}(t)]\\
&+\int_{0}^{t}e^{-c_k(t-s)}\mathcal{F}_{b_q}[F_{1}(s)]ds-\int_{0}^{t}e^{-c_k(t-s)}\mathcal{F}_{b_q}[F_{3}(s)]ds-\mathcal{F}_{a_q}[F_{3}(t)],\\
\mathcal{F}_{a_q}(\partial_t\rho(t))&\lesssim e^{-c_kt}[\frac{\lambda_{-}\lambda_{+}}{C_0(q\frac{\pi}{2})^{2}(-\sqrt{\delta})}\mathcal{F}_{b_q}(u_{10})
-\frac{\lambda_{-}}{\sqrt{\delta}}\mathcal{F}_{a_q}(u_{20})]+\mathcal{F}_{b_q}[F_{1}(t)]\\
&+\int_{0}^{t}e^{-c_k(t-s)}\mathcal{F}_{b_q}[F_{1}(s)]ds-\int_{0}^{t}e^{-c_k(t-s)}\mathcal{F}_{a_q}[F_{3}(s)]ds-\mathcal{F}_{a_q}[F_{3}(t)].
\end{aligned}
\end{equation}

\section{The proof of the main results}

Recalling the Proposition 3.2 and Proposition 3.3, it is sufficient to get the bound of $\int_0^TE_4dt$. We shall prove this bound by a bootstrap procedure.
For simplicity, when $k\geq 7$, we denote
\begin{equation}
\|\bar{\Phi}_0\|_{H^{k+3}}+\|\Psi_0\|_{H^{k+3}}\triangleq M_0,
\end{equation}
we have the following decay estimates:
\begin{lemma}
Assuming that $E_k^2+\Gamma_{k+1}^2\lesssim\varepsilon_0^2$, for  $t\in [0,T]$ and  $k\geq7$, then we have
\begin{equation}
E_4\leq 4M_0 e^{-\beta t},
\end{equation}
with $\beta=\min\{\alpha, \frac{c_\kappa}2\}$.
\end{lemma}
\begin{proof}
From \eqref{4.21}, \eqref{4.37} and \eqref{4.38}, by using the Plancherel's theorem,  we have
\begin{equation}
E_4(t)\lesssim e^{-c_kt}E_4(0)+\int_0^te^{-c_k(t-s)}\|F(\cdot,s)\|_{H^5}ds+\|F(\cdot,t)\|_{H^5}.
\end{equation}

Noting that $k\geq 7$ and $E_k^2+\Gamma_{k+1}^2\lesssim\varepsilon_0^2$,   we have
\begin{equation}
\|F(\cdot,t)\|_{H^5}\lesssim  E_6 E_4+E_6(\|\bar{\Phi}\|_{H^7}+\|\Psi\|_{H^6}).
\end{equation}
By using Proposition 2.2, there exists a  uniform constant $C_1$, such that
\begin{eqnarray}\label{5.5}
&&E_4(t)\leq C_1 e^{-c_kt}\epsilon_0^2+C_1\big(E_6 E_4+E_6(\|\bar{\Phi}\|_{H^7}+\|\Psi\|_{H^6})\big)\\\nonumber
&&+C_1\int_0^te^{-c_k(t-s)}(E_4E_k+E_k(\|\bar{\Phi}\|_{H^7}+\|\Psi\|_{H^6}))ds\\\nonumber
&&\leq C_1\bigg( e^{-c_kt}\epsilon_0^2+  \varepsilon_0 \big(E_4+M_0e^{-\alpha t}\big)+ \int_0^t \varepsilon_0 e^{-c_k(t-s)}(E_4+ M_0e^{-\alpha s})ds\bigg).
\end{eqnarray}

Noting $\beta=\min\{\alpha, \frac{c_\kappa}2\}$, by using the continuity procedure, we  now claim that
\begin{equation}\label{5.6}
E_4(t)\leq 4M_0 e^{-\beta t}.
\end{equation}

Actually,  from \eqref{5.5}, we have
\begin{eqnarray}\label{5.7}
&&\int_0^t \varepsilon_0 e^{-c_{\kappa}(t-s)}(E_4+ M_0e^{-\alpha s})ds\leq M_0\varepsilon_0e^{-c_{\kappa} t} \int_0^t   e^{ c_{\kappa}s}(4e^{-\beta s}+ e^{-\alpha s})ds
\\\nonumber
&&\lesssim M_0 \varepsilon_0 e^{-c_{\kappa}t}\big( e^{(c_\kappa-\beta) t}+e^{(c_\kappa-\alpha )t}\big)+ M_0 \varepsilon_0 e^{-c_{\kappa}t}
\leq C_2  M_0 \varepsilon_0 e^{-\beta t},
\end{eqnarray}
with $C_2$ be a uniform constant.

Then from \eqref{5.5} and \eqref{5.7}, we get
\begin{equation}\label{5.8}
E_4(t)\leq e^{-\beta t}\varepsilon_0\big( C_1\varepsilon_0+ 5C_1M_0 +C_1C_2 M_0\big).
\end{equation}

Noting that  $\varepsilon_0$ is  small enough, by taking  $\varepsilon_0\big( C_1\varepsilon_0+ 5C_1M_0 +C_1C_2 M_0\big)< 2M_0$,
which proves our claim \eqref{5.6}.
\end{proof}

{\bf Proof of our main theorem.}

From the Proposition 2.2,  Proposition 3.2 and Proposition 3.3, by using the Gronwall's inequality and Lemma 5.1, we have
\begin{eqnarray}\label{5.9}
&& E_k^2(t)+\Gamma_{k+1}^2(t)\\\nonumber
&&\lesssim (E_k(0)^2+\Gamma_{k+1}^2(0))e^{ \int_0^tE_4+\|\bar{\Phi}\|_{H^{k+3}} +\|\Psi\|_{H^{k+2}} +\|\partial_t{\Psi}\|_{H^{k+2}} ds }\lesssim \epsilon_0^2 e^{(\frac4{\beta}+\frac{1}{\alpha})M_0}.
\end{eqnarray}

Thus by \eqref{5.9}, we have
\begin{eqnarray}\label{5.10}
\|u\|_{H^k}^2+\|\partial_{t}u\|_{H^k}^2+\kappa\int_0^T\|u\|_{H^k}^2+\|\partial_{t}u\|_{H^k}^2ds\lesssim \varepsilon_0^{2}
\end{eqnarray}
from which we complete our proofs.

\section*{Acknowledgement}  Yi Du and Wang Yang were supported
by NSFC (grant No. 11471126; 11971199). Yi Zhou was supported by Key Laboratory of Mathematics for Nonlinear Sciences (Fudan University), Ministry of Education of China, P.R.China.
Shanghai Key Laboratory for Contemporary Applied Mathematics, School of Mathematical Sciences, Fudan University, P.R. China, NSFC (grants No. 11421061), 973 program (grant No. 2013CB834100) and 111 project.

\end{document}